\documentclass[12pt,a4paper]{amsart}

\usepackage{amssymb}
\usepackage[numeric, abbrev, nobysame]{amsrefs}
\usepackage{amscd}

\def\frak{\mathfrak}
\def\Bbb{\mathbb}
\def\Cal{\mathcal}

\let\phi\varphi

\newcommand{\x}{\times}
\renewcommand{\o}{\circ}

\newcommand{\al}{\alpha}
\newcommand{\be}{\beta}
\newcommand{\ga}{\gamma}

\newcommand{\ka}{\kappa}
\newcommand{\la}{\lambda}
\newcommand{\om}{\omega}
\newcommand{\ph}{\phi}
\newcommand{\ps}{\psi}
\renewcommand{\th}{\theta}
\newcommand{\si}{\sigma}
\newcommand{\ze}{\zeta}
\newcommand{\Ga}{\Gamma}
\newcommand{\La}{\Lambda}
\newcommand{\Ph}{\Phi}
\newcommand{\Ps}{\Psi}
\newcommand{\Om}{\Omega}

\newcommand{\im}{\operatorname{im}}
\newcommand{\Ric}{\operatorname{Ric}}

\newcommand{\Alt}{\operatorname{Alt}}
\newcommand{\Ad}{\operatorname{Ad}}

\newcommand{\Euc}{\operatorname{Euc}}
\newcommand{\euc}{\mathfrak{euc}}

\newcommand{\tn}{\widetilde{\nabla}}
\newcommand{\hor}{\text{hor}}
\newcommand{\Fl}{\operatorname{Fl}}

\newcounter{theorem}
\numberwithin{theorem}{section}
\numberwithin{equation}{section}

\newtheorem{thm}[theorem]{Theorem}
\newtheorem*{thm*}{Theorem \thesubsection}
\newtheorem{lemma}[theorem]{Lemma}
\newtheorem{prop}[theorem]{Proposition}

\newtheorem*{lemma*}{Lemma \thesubsection}
\newtheorem*{prop*}{Proposition \thesubsection}
\newtheorem*{cor*}{Corollary \thesubsection}

\theoremstyle{definition}
\newtheorem{definition}[theorem]{Definition}
\newtheorem*{definition*}{Definition \thesubsection}

\newtheorem*{example*}{Example \thesubsection}
\theoremstyle{remark}
\newtheorem{remark}[theorem]{Remark}
\newtheorem*{remark*}{Remark \thesubsection}

\def\sideremark#1{\ifvmode\leavevmode\fi\vadjust{\vbox to0pt{\vss
 \hbox to 0pt{\hskip\hsize\hskip1em
 \vbox{\hsize3cm\tiny\raggedright\pretolerance10000
  \noindent #1\hfill}\hss}\vbox to8pt{\vfil}\vss}}}%
                        
\begin{document}
\renewcommand{\today}{} 

\title{A new construction of the\\ Riemannian deformation sequence }

\author{Andreas \v Cap}

\address{Faculty of Mathematics\\
  University of Vienna\\
  Oskar--Morgenstern--Platz 1\\
  1090 Wien\\
  Austria} 

\email{Andreas.Cap@univie.ac.at}

\date{July 22, 2026}



\begin{abstract}
  We obtain a new construction of a sequence of invariant differential operators on a
  Riemannian manifold $(M,g)$ that governs the linearized deformation theory of
  $g$. Starting from an explicit linear connection on a natural bundle $\Cal AM\to M$
  we construct a twisted de Rham sequence and then apply an analog of the
  construction of BGG sequences. If $g$ has constant sectional curvature, both
  sequences are complexes which compute the cohomology of the sheaf of local Killing
  fields, which are equivalent to parallel sections of $\Cal AM$.

  In a second step, we relate the construction to the description of $(M,g)$ as a
  (torsion-free) Cartan geometry $(\Cal OM,\om)$, where $\Cal OM$ is the orthonormal
  frame bundle of $M$. This provides a manifest relation of the twisted de Rham
  sequence to the deformation theory of the Cartan connection $\om$ (which is easier
  do deal with than the deformation theory of $g$). The BGG-like construction can
  then be nicely viewed as interpreting the linearized deformation theory of torsion
  free Cartan geometries in terms of the underlying Riemannian metric.
\end{abstract}

\maketitle

\pagestyle{myheadings}\markboth{Andreas \v Cap}{Riemannian deformation sequence}

\section{Introduction}\label{1}
The Riemannian deformation sequence is a sequence of natural differential operators
between sections of natural vector bundles over a Riemannian manifold $(M,g)$, which
governs the linearized deformation theory of $g$. It starts with the Killing operator
on vector fields, whose kernel consists of Killing vector fields, which are the
infinitesimal isometries of $g$. The next operator is related to the infinitesimal
change of curvature caused by an infinitesimal change of metric. The standard
reference for a construction of such a sequence is the 1961 article \cite{Calabi} of
E.\ Calabi. The construction there focuses on the case that $g$ has constant
sectional curvature, in which one obtains a complex which is shown to provide a fine
resolution of the sheaf of local Killing vector fields on $(M,g)$. Calabi's
construction (described as a ``tour de force of mathematical computation'' in
L.\ Auslander's MR review of the article) is based on the method of moving
frames. The definition of the operators in Calabi's sequence is motivated by
deformation theory and the main application is to computing the cohomology of the
sheaf of local Killing vector fields (in the constant-curvature case), which has a
direct interpretation in terms of deformation theory.

Calabi's construction was discussed and extended in the articles \cite{Rag} by
M.S.\ Raghunathan and \cite{BBL} by L.\ B\'erard-Bergery, J.-P.\ Burguignon and
J.\ Lafontaine, which unfortunately both seem not to be well-known. A major topic
building on Calabi's work is the study of deformations of symmetric spaces which was
initiated by J.\ Gasqui and H.\ Goldschmidt in \cite{GG} and continues to be a very
active area, see for example \cite{CELM} and the references therein. 

Apart from the obvious interest for Riemannian geometry, the deformation sequence for
domains in $\Bbb R^n$ with the flat metric, known as the \textit{Calabi complex}, is
also important in applied mathematics. This comes from its role in linear elasticity
theory, see e.g.\ \cite{AFW}, whence it is known as the \textit{elasticity complex}
in this area. The quest for efficient numerical methods for elasticity led to an
intense study of analytical properties of the Calabi complex in the applied math
community. While the connection to Riemannian geometry is visible in the elasticity
complex, the relation to deformation theory is of minor importance.

In the late 1990's and early 2000's a very surprising aspect of Calabi's sequence was
discovered. It turned out that in an appropriate interpretation, it can be obtained
as an instance of a Bernstein-Gelfand-Gelfand sequence (or BGG sequence) for the
underlying projective structure of the Riemannian metric $g$. This projective
structure is the equivalence class of all linear connections on the tangent bundle
$TM$ which have the same geodesics as the Levi-Civita connection of $g$ up to
parametrization. In an appropriate interpretation (involving densities), the
operators in the sequence have the same formula in terms of any of these
connections. This is discussed in \cite{Mike-Elasticity} and can be interpreted as a
stability result on the operators parallel to conformal invariance of differential
operators on Riemannian manifolds. Projective invariance gives a nice explanation for
the fact that one obtains a complex exactly in the case of constant-curvature
metrics, since by a classical result of Beltrami these are exactly the metrics for
which the underlying projective structure is flat. See also the article
\cite{Khavkine} which discusses the perspective of representation theory on Calabi's
constructions (without using BGG methods). This result has led to considerable
interest in the BGG construction and similar methods in the applied community.

Projective structures belong to the large class of geometric structures known as
\textit{parabolic geometries}, see \cite{book}, and there is a general construction
of BGG sequences for all these geometric structures originally developed in
\cite{CSS-BGG}. Besides the projective structure, a Riemannian structure determines
a second underlying parabolic geometry, namely a conformal structure. Hence
projective and conformal BGG sequences can be constructed on Riemannian manifolds and
things simplify considerably in this setting, see \cite{Riem-BGG} for a discussion.

Apart from the general background, the construction as a BGG sequence has immediate
conceptual advantages. BGG sequences are always constructed from twisted de Rham
sequences of differential forms with values in some vector bundle endowed with a
linear connection. In case this connection is flat, the de Rham sequence is well
known to be a complex and general results imply that then also the BGG sequence is a
complex which computes the same cohomology. Since all the constructions are local, it
immediately follows that then the BGG sequence is a fine resolution of the sheaf
defined by the kernel of the first operator in the sequence. It also turns out that
this sheaf is isomorphic to local parallel sections of the vector bundle used in the
initial twisted de Rham sequence. In contrast, proving that one obtains a resolution
takes up a substantial part of \cite{Calabi}.

There is an important drawback, both with Calabi's original construction and with the
construction as a projective BGG sequence. While the construction of the operators in
Calabi's sequence is motivated by deformation theory, there is no systematic
connection of the sequence to deformation theory. In particular, the infinitesimal
change of curvature induced by an infinitesimal change of metric is computed directly
in general, see p.\ 161 of \cite{Calabi}. Then the version for constant curvature in
equation (12) of the article used as a motivation for the definition of the second
operator in the sequence (see p.\ 170). Unfortunately, there is a sign error in the
passage from the general formula to the constant curvature version in formula
(12). So while all the results in \cite{Calabi} are correct, the interpretation of
the second operator as computing the infinitesimal change of curvature caused by an
infinitesimal change of metric is not. See Section 6.4 of \cite{CELM} for a
discussion of the difference between the two operators.  This problem was resolved in
\cite{BBL}, where the sequence involves the operator that computes the infinitesimal
change of curvature caused by an infinitesimal deformation of the metric. Still the
relation to deformations is only through explicit formulae in \cite{BBL}.

For the construction as a projective BGG sequence, the situation is even worse. In
the BGG version the sequence starts with a space of (weighted) one-forms, so to
identify these with vector fields as needed for deformation interpretations, one
already has to use the Riemannian metric and leave the projectively invariant
setting.  Even on the level of the Killing operator, the relation to deformation
theory is only available on a computational level. Since the BGG construction
recovers Calabi's sequence, its second operator also does not compute the
infinitesimal change of curvature caused by an infinitesimal change of metric (except
for flat metrics).

The main aim of this article is to provide an alternative construction of a
Riemannian deformation sequence which combines the advantages of the BGG construction
with a conceptual relation to deformation theory. This is based on the description of
Riemannian manifolds as Cartan geometries and similar to the construction in the
(much more complicated) situation of parabolic geometries in \cite{deformations}. We
give a direct construction of the sequence in Section \ref{2}. Given a Riemannian
manifold $(M,g)$, we consider the bundle $\frak o(TM)$ of endomorphisms of the
tangent bundle $TM$, which are skew symmetric with respect to $g$. We then define two
linear connections $\nabla^{\Cal A}$ and $\tilde\nabla^{\Cal A}$ on the bundle $\Cal
AM:=TM\oplus\frak o(TM)$ and compute their curvatures. In particular,
$\tilde\nabla^{\Cal A}$ is a flat connection if and only if $g$ has constant
sectional curvature. These connections can be extended to operators on differential
forms with values in $\Cal AM$, thus defining twisted de Rham sequences. If $g$ has
constant sectional curvature, then the sequence coming from $\tilde\nabla^{\Cal A}$
provides a fine resolution of the sheaf of local parallel sections of $\Cal AM$.

We next apply an analog of the simplified BGG construction from \cite{C-H:weak} and
\cite{Riem-BGG} to the twisted de Rham sequence defined by $\tilde\nabla^{\Cal
  A}$. This leads to natural differential operators $D_i$ acting between subbundles
$\Cal H_iM$ of $\Cal AM$. On the way, we construct splitting operators mapping
$\Ga(\Cal H_iM)$ to $\Om^i(M,\Cal AM)$ and we prove in Theorem \ref{thm2.5} that they
intertwine between the operators in the two sequences. We also show that if the
twisted de Rham sequence is a complex then so is the BGG-sequence $(\Ga(\Cal
H_*M),D_*)$ and both complexes compute the same cohomology. This also implies that
the sheaf of local parallel sections of $\Cal AM$ is canonically identified with the
sheaf of local Killing fields. From the construction, it is easy to derive explicit
formulae for the operators in the sequence.

The connection to deformation theory is made in Section \ref{3}. We start by briefly
reviewing the equivalent description of Riemannian manifolds as Cartan
geometries. The relation of the Cartan connection to the Levi-Civita connection
quickly leads to an interpretation of the connections $\nabla^{\Cal A}$ and
$\tilde\nabla^{\Cal A}$ in terms of the Cartan geometry. In particular,
$\tilde\nabla^{\Cal A}$ is induced by the Lie derivative of the Cartan connection,
which directly provides a connection to deformation theory. This leads to an
interpretation of the twisted de Rham sequence defined by $\tilde\nabla^{\Cal A}$ as
a deformation sequence on the level of Cartan geometries. The BGG construction
discussed above can then be interpreted in this picture as restricting to
torsion-free deformations and then passing from the Cartan geometry picture to the
picture of the underlying Riemannian metric. This in particular shows that the second
operator $D_1$ in the BGG sequence constructed in Section \ref{2} indeed computes the
infinitesimal change of curvature caused by an infinitesimal deformation of the
metric. The bundles $\Cal H_iM$ arise as associated bundles modelled on certain Lie
algebra cohomology spaces, a conceptual explanation of why this leads to the bundles
occurring in a projective BGG sequence is given in Remark \ref{rem3.5} (2).

\subsection*{Acknowledgements} I would like to thank Michael Eastwood for drawing my
attention to the issues with Calabi's construction in \cite{Calabi}, which motivated
me to start the work resulting in this article and for very helpful subsequent
discussions. I would also like to thank J.-P.\ Bourguignon, who pointed me to the
article \cite{BBL} and provided valuable input. This article is based upon work from
COST Action CaLISTA CA21109 supported by COST (European Cooperation in Science and
Technology). https://www.cost.eu.

\section{Direct construction of a deformation sequence}\label{2}
We start with an ad-hoc definition of a linear connection on a natural vector bundle
$\Cal AM$ over a Riemannian manifold $(M,g)$ and prove that it is flat if and only if
$g$ has constant sectional curvature. We then apply an analog of the construction of
BGG sequences to the induced twisted de Rham sequence of $\Cal AM$-valued
differential forms. This produces a sequence of natural differential operators acting
on sections of certain natural vector bundles over $M$, which we call a BGG
sequence. If $g$ has constant sectional curvature, the twisted de Rham sequence is a
fine resolution of the sheaf of local parallel sections of $\Cal AM$ and we easily
conclude that also the BGG sequence is a resolution of this sheaf, which is
isomorphic to the sheaf of kernels of its first operator. Explicit formulae for the
operators in the BGG sequence can be easily computed directly and this shows that the
first two operators in the sequence are the Killing operator and the operator
computing the infinitesimal change of curvature caused by an infinitesimal
deformation of the metric, respectively. Thus we recover the results of \cite{Calabi}
and \cite{BBL} with proofs that require much less computation and are more
conceptual, since the property of being a fine resolution is deduced directly from
the corresponding property of the twisted de Rham sequence. This also shows that the
relation to Cartan geometries that motivates the construction and is discussed in
Section \ref{3} below is not needed on a formal level.

\subsection{Some background on Riemannian geometry}\label{2.1}
Recall that natural vector bundles on Riemannian manifolds of dimension $n$ are in
bijective correspondence with representations of the Lie group $O(n)$ and
$O(n)$-equivariant linear maps between such representations induce natural bundle
maps between the corresponding bundles. This can most easily understood via forming
associated bundles to the orthonormal frame bundle, i.e.\ in the setting discussed in
Section \ref{3}. Since the Levi-Civita connection on a Riemannian manifold defines a
principal connection on the orthonormal frame bundle, there is an induced linear
connection on any vector bundle. We will refer to all these connections as the
Levi-Civita connection and denote them by $\nabla$. This description also immediately
implies that the bundle maps induced by $O(n)$-equivariant maps are parallel for the
Levi-Civita connection.

An important example is provided by the adjoint representation of $O(n)$ on its Lie
algebra $\frak o(n)$. On a Riemannian manifold $(M,g)$ of dimension $n$, the standard
representation of $O(n)$ on $\Bbb R^n$ corresponds to the tangent bundle $TM$, so the
bundle corresponding to the adjoint representation $\frak o(n)$ is the bundle $\frak
o(TM)$ of endomorphisms of $TM$, which are skew symmetric with respect to $g$.  Thus,
a $\binom11$-tensor field $\Ph$ is a section of the subbundle $\frak o(TM)\subset
T^*M\otimes TM$ if and only if the $\binom02$-tensor field $(\xi,\eta)\mapsto
g(\xi,\Ph(\eta))$ is skew symmetric.

Now for any representation $\Bbb W$ of $O(n)$, the infinitesimal representation of
$\frak o(n)$ can be interpreted as an $O(n)$-equivariant bilinear map $\frak
o(n)\times\Bbb W\to\Bbb W$. Denoting by $WM$ the natural vector bundle induced by
$\Bbb W$ for $(M,g)$, this induces a bilinear bundle map $\frak o(TM)\x WM\to WM$. We
denote all these bundle maps as well as the induced tensorial map on sections by
$\bullet$. So for $\Phi\in\Ga(\frak o(TM))$ and $\si\in\Ga(WM)$, we obtain
$\Phi\bullet\si\in\Ga(WM)$. The fact that $\bullet$ is parallel for the Levi-Civita
connection says that for any $\xi\in\frak X(M)$, we obtain
\begin{equation}\label{bull-diff}
\nabla_\xi (\Phi\bullet\si)=(\nabla_\xi\Phi)\bullet\si+\Phi\bullet\nabla_\xi\si. 
\end{equation}

The bundle maps $\bullet$ give a neat description of the curvature of the Levi-Civita
connection on a general natural vector bundle $WM$. Recall that the well-known
symmetries of the Riemann curvature tensor (which by definition encodes the curvature
of the Levi-Civita connection on $TM$) is an element $R\in\Om^2(M,\frak o(TM))$. For
the curvature of the Levi-Civita connection on $WM$, one then obtains for
$\xi_1,\xi_2\in\frak X(M)$ and $\si\in\Ga(WM)$ the equation
\begin{equation}\label{curv-W}
  \nabla_{\xi_1}\nabla_{\xi_2}\si-\nabla_{\xi_2}\nabla_{\xi_1}\si-\nabla_{[\xi_1,\xi_2]}\si=
  R(\xi_1,\xi_2)\bullet\si.
\end{equation}
Finally, we recall the two Bianchi identities that the Riemann curvature tensor $R$
satisfies. On the one hand, the first, or algebraic, Bianchi identity says that for
$\xi_1,\xi_2,\xi_3\in\frak X(M)$ one gets
\begin{equation}\label{Bianchi1}
R(\xi_1,\xi_2)(\xi_3)+R(\xi_3,\xi_1)(\xi_2)+R(\xi_2,\xi_3)(\xi_1)=0. 
\end{equation}
The second, or differential, Bianchi identity concerns the covariant derivative of the
Riemann curvature. For $\xi\in\frak X(M)$, the covariant derivative $\nabla_\xi R$ is
again an element of $\Om^2(M,\frak o(TM))$ and for $\xi_1,\xi_2,\xi_3\in\frak X(M)$,
one gets
\begin{equation}\label{Bianchi2}
(\nabla_{\xi_1}R)(\xi_2,\xi_3)+(\nabla_{\xi_3}R)(\xi_1,\xi_2)+(\nabla_{\xi_2}R)(\xi_3,\xi_1)=0.
\end{equation}

\subsection{Riemannian adjoint tractors}\label{2.2}
We introduce the central objects for this article here in an ad-hoc manner and verify
their properties by direct computations. The motivation for considering exactly these
objects and conceptual explanation for their properties will be given in Section
\ref{3.3} below, see in particular Theorem \ref{thm3.3}.

Let $(M,g)$ be a Riemannian manifold of dimension $n\geq 2$. We define the Riemannian
adjoint tractor bundle $\Cal AM:=TM\oplus\frak o(TM)$. We will write sections of
$\Cal AM$ as vectors $\binom{\eta}{\Phi}$ for a vector field $\eta\in\frak X(M)$ and
a section $\Ph$ of $\frak o(TM)$. We can use the Levi-Civita connection $\nabla$ of
$g$ to differentiate sections of $\Cal AM$ component-wise. Next we define two
modifications of this linear connection on $\Cal AM$.

\begin{definition}\label{def2.2}
We define the \textit{adjoint tractor connection} $\nabla^{\Cal A}$ and the
\textit{deformation connection} $\tn^{\Cal A}$ by the following formulae with
$\xi,\eta\in\frak X(M)$ and $\Phi\in\Ga(\frak o(TM))$:
\begin{equation}\label{conn-def}
\nabla^{\Cal A}_\xi\begin{pmatrix} \eta\\ \Phi \end{pmatrix}:=\begin{pmatrix}
\nabla_\xi\eta-\Phi(\xi)\\ \nabla_\xi\Phi \end{pmatrix} \qquad 
\tn^{\Cal A}_\xi\begin{pmatrix} \eta\\ \Phi \end{pmatrix}:=\begin{pmatrix}
\nabla_\xi\eta-\Phi(\xi)\\ \nabla_\xi\Phi-R(\xi,\eta) \end{pmatrix}
\end{equation}
\end{definition}

The basic properties of these operations are easy to obtain.

\begin{prop}\label{prop2.2}
Both $\nabla^{\Cal A}$ and $\tn^{\Cal A}$ define linear connections on the bundle
$\Cal AM$ that are canonically associated to $(M,g)$. Denoting their curvatures by
$R^{\Cal A}$ and $\widetilde{R}^{\Cal A}$, respectively, we obtain
$$
R^{\Cal A}(\xi_1,\xi_2)\begin{pmatrix} \eta\\ \Phi \end{pmatrix}=\begin{pmatrix}
R(\xi_1,\xi_2)(\eta) \\ R(\xi_1,\xi_2)\bullet \Phi \end{pmatrix}
$$

$$ \widetilde{R}^{\Cal A}(\xi_1,\xi_2)\begin{pmatrix}
  \eta\\ \Phi \end{pmatrix}=\begin{pmatrix} 0 \\ (\nabla_\eta
R)(\xi_1,\xi_2)-(\Phi\bullet R)(\xi_1,\xi_2)\end{pmatrix}
$$
In particular, the connection $\nabla^{\Cal A}$ is flat if and only if $g$ is flat,
while $\tn^{\Cal A}$ is flat if and only if $g$ has constant sectional curvature.
\end{prop}
\begin{proof}
Both expressions in \eqref{conn-def} are the sum of the component-wise Levi-Civita
connection and a tensorial term, so they both define linear connections on $\Cal
AM$. To compute the curvatures we directly evaluate the defining equation for the
curvature of a linear connection. Starting with $\nabla^{\Cal A}$, differentiating the
first formula in \eqref{conn-def} shows that we obtain for $\xi_1,\xi_2,\eta\in\frak
X(M)$ and $\Ph\in\frak o(TM)$ the equation
\begin{equation}\label{n2}
\nabla^{\Cal A}_{\xi_1}\nabla^{\Cal
  A}_{\xi_2}\begin{pmatrix}\eta\\ \Phi \end{pmatrix}=
\begin{pmatrix}
  \nabla_{\xi_1}\nabla_{\xi_2}\eta-\nabla_{\xi_1}\Phi(\xi_2)-(\nabla_{\xi_2}\Phi)(\xi_1)\\
  \nabla_{\xi_1}\nabla_{\xi_2}\Phi
\end{pmatrix}. 
\end{equation}
To obtain $R^{\Cal A}(\xi_1,\xi_2)$ we have to subtract from this the same expression
with $\xi_1$ and $\xi_2$ exchanged as well as
\begin{equation}\label{brack-term}
  \begin{pmatrix} \nabla_{[\xi_1,\xi_2]}\eta-\Ph([\xi_1,\xi_2])\\
    \nabla_{[\xi_1,\xi_2]}\Phi\end{pmatrix}.
\end{equation}    
In the $\frak o(TM)$-component, this readily produces $R(\xi_1,\xi_2)\bullet \Phi$ by
\eqref{curv-W}. In the $TM$-component, we obtain the sum of $R(\xi_1,\xi_2)(\eta)$
and
\begin{equation}\label{nb}
-\nabla_{\xi_1}\Phi(\xi_2)-(\nabla_{\xi_2}\Phi)(\xi_1)+\nabla_{\xi_2}\Phi(\xi_1)+
(\nabla_{\xi_1}\Phi)(\xi_2)-\Phi([\xi_1,\xi_2]). 
\end{equation}
Now the first and fourth term add up to $\Ph(\nabla_{\xi_1}\xi_2)$, while the second
and third term add up to $-\Ph(\nabla_{\xi_2}\xi_1)$. But torsion-freeness of the
Levi-Civita connection reads as
$\nabla_{\xi_1}\xi_2-\nabla_{\xi_2}\xi_1=[\xi_1,\xi_2]$, which completes the proof of
the formula for $R^{\Cal A}$. In particular, $R^{\Cal A}$ vanishes if and only if $R$
vanishes which happens if and only if $g$ is flat.

Proceeding to $\tn^{\Cal A}$, differentiating the second formula in \eqref{conn-def},
we conclude that to compute $\tn^{\Cal A}_{\xi_1}\tn^{\Cal
  A}_{\xi_2}\binom{\eta}{\Phi}$, we have to add
\begin{equation}\label{tn2}
\begin{pmatrix} R(\xi_2,\eta)(\xi_1) \\ -\nabla_{\xi_1}R(\xi_2,\eta) -
  R(\xi_1,\nabla_{\xi_2}\eta)+R(\xi_1,\Phi(\xi_2)) 
\end{pmatrix}
\end{equation}
  to the right hand side of \eqref{n2}. Likewise, to compute $\tn^{\Cal
    A}_{[\xi_1,\xi_2]}\binom{\eta}{\Phi}$, we have to subtract
  $R([\xi_1,\xi_2],\eta)$ from the bottom component of the right hand side of
  \eqref{brack-term}.

  Together with the description of $R^{\Cal A}$, this readily implies that the top
  component of $\widetilde{R}^{\Cal A}(\xi_1,\xi_2)\binom{\eta}{\Phi}$ is given by
  $$
  R(\xi_1,\xi_2)(\eta)+R(\xi_2,\eta)(\xi_1)-R(\xi_1,\eta)(\xi_2)
  $$
  which vanishes by the first Bianchi identity \eqref{Bianchi1}. In the bottom
  component, we have terms containing $\eta$ and terms containing $\Ph$ and we treat
  these terms separately. For the terms containing $\eta$, we can first use
  torsion-freeness of $\nabla$ to write the contribution from the bracket term as
  $$
  R([\xi_1,\xi_2],\eta)=R(\nabla_{\xi_1}\xi_2,\eta)-R(\nabla_{\xi_2}\xi_1,\eta). 
  $$
  Adding this to the contributions from \eqref{tn2}, we can use the Leibniz rule for
  $\nabla$ to write the result as
  $$
  -(\nabla_{\xi_1}R)(\xi_2,\eta)+(\nabla_{\xi_2}R)(\xi_1,\eta)=(\nabla_\eta
  R)(\xi_1,\xi_2). 
  $$
  Here we used the second Bianchi identity \eqref{Bianchi2} in the last step.

  To understand the terms involving $\Phi$, we have to interpret the term
  $R(\xi_1,\xi_2)\bullet\Ph\in\Ga(\frak o(TM))$ taking into account that the
  infinitesimal action of $\frak o(n)$ on itself is the adjoint action. Hence for
  $\zeta\in\frak X(M)$, we obtain
  $$
(R(\xi_1,\xi_2)\bullet\Ph)(\zeta)=R(\xi_1,\xi_2)(\Phi(\zeta))-\Ph(R(\xi_1,\xi_2)(\zeta)). 
  $$
  On the other hand, the additional terms from \eqref{tn2} map $\zeta$ to
  $$
  R(\xi_1,\Ph(\xi_2))(\zeta)+R(\Phi(\xi_1),\xi_2)(\zeta), 
  $$ and adding this to the above, we exactly obtain $-(\Phi\bullet
  R)(\xi_1,\xi_2,\zeta)$ and we obtain the claimed formula for $\widetilde{R}^{\Cal
    A}$. Of course, this vanishes if and only if $\nabla R=0$ and $\Phi\bullet R=0$
  for any $\Phi\in\Ga(\frak (TM))$. The second equation means that in each point
  $x\in M$, the value $R(x)\in\La^2T^*_xM\otimes\frak o(T_xM)$ lies in the kernel of
  the action of $\frak o(T_xM)$, which has a direct interpretation in terms of
  representation theory. Indeed, decomposing $R$ into the Weyl curvature $W$ and the
  Ricci tensor $\Ric$, it follows that $W(x)=0$ and $\Ric(x)=\la(x)g_x$. Hence
  $(M,g)$ has to be conformally flat and Einstein, which is well known to imply that
  $\lambda$ is constant and to be equivalent to $(M,g)$ having constant sectional
  curvature. Since constant sectional curvature also implies that $\nabla R=0$, this
  completes the proof.
\end{proof}

\subsection{Covariant exterior derivatives}\label{2.3}
It is a general fact, that a linear connection $\nabla$ on a vector bundle $E\to M$
can be coupled to the exterior derivative to obtain operators $d^{\nabla}$ defined on
differential forms with values in $E$. More explicitly, for $k=0,\dots,n=\dim(M)$ we
define $\Om^k(M,E):=\Ga(\La^kT^*M\otimes E)$, so elements of this space can be viewed
as $k$-linear alternating operators which send $k$ vector fields to a section of $E$
and are linear over smooth functions in each argument. In this language, given
$\ph\in\Om^k(M,E)$ and vector fields $\xi_0,\dots,\xi_k\in\frak X(M)$ and denoting by
hats an omission of arguments, one defines
\begin{equation}\label{dn-def}
  \begin{aligned}
    d^{\nabla}\ph(\xi_0,\dots,\xi_k)&=\textstyle\sum_{i=0}^k(-1)^i\nabla_{\xi_i}
    \ph(\xi_1,\dots,\widehat{\xi_i},\dots,\xi_k)\\ &+
  \textstyle\sum_{i<j}(-1)^{i+j}\ph([\xi_i,\xi_j],\xi_1,
  \dots,\widehat{\xi_i},\dots,\widehat{\xi_j},\dots,\xi_k).
  \end{aligned}
\end{equation}
An easy computation shows that this expression is linear over smooth functions in
each argument and hence defines $d^{\nabla}\ph\in\Om^{k+1}(M,E)$. In particular, we
can apply this to the Levi-Civita connection on any natural vector bundle to obtain
operators $d^{\nabla}$ and to the two connections from Proposition \ref{prop2.2} to
obtain operators $d^{\nabla^{\Cal A}}$ and $d^{\tn^{\Cal A}}$ that act on
$\Om^*(M,\Cal AM)$. To describe the relation between these operators, we have to
introduce some notation and simple operations.

An element of $\Om^k(M,\Cal AM)$ can be written as a vector $\binom{\ps}{\Psi}$ with
$\ps\in\Om^k(M,TM)$ and $\Psi\in\Om^k(M,\frak o(TM))$. By definition, the form $\Psi$
associates to $k$ vector fields a section of $\frak o(TM)$, which again can be viewed
as a tensorial operator mapping vector fields to vector fields. Hence $\Psi$ is a
special $\binom1{k+1}$ tensor field, which suggests and one can naturally assign to
this an element $\Alt(\Psi)\in\Om^{k+1}(M,TM)$ defined by
\begin{equation}\label{Alt-def}
  \Alt(\Psi)(\xi_0,\dots,\xi_k):=
  \textstyle\sum_{i=0}^k(-1)^{k-i}\Psi(\xi_0,\dots,\widehat{\xi_i},\dots,\xi_k)(\xi_i). 
\end{equation}
Since $\Psi$ is alternating in its $k$ arguments it easily follows that $\Alt(\Psi)$
is alternating in all $k+1$ arguments and hence defines a form of the required
type. On the other hand, given a form $\ps\in\Om^k(M,TM)$ we can use the Riemann
curvature tensor $R$ to define $i_\ps R\in\Om^{k+1}(M,\frak o(TM))$ by
\begin{equation}\label{i-def}
(i_\ps R)(\xi_0,\dots,\xi_k):=
  \textstyle\sum_{i=0}^k(-1)^{k-i}R(\psi(\xi_0,\dots,\widehat{\xi_i},\dots,\xi_k),\xi_i). 
\end{equation}
Again, the fact that $\psi$ is alternating in its $k$ arguments implies that
$i_{\psi}R$ is alternating in its $k+1$ arguments and since $R$ has values in $\frak
o(TM)$, we again obtain a form of the required type. Using this, we now formulate

\begin{prop}\label{prop2.3}
  In terms of the operations defined above, we get
  $$
  d^{\nabla^{\Cal A}}\begin{pmatrix}\psi\\\Psi
  \end{pmatrix}=\begin{pmatrix} d^{\nabla}\psi-(-1)^k\Alt(\Psi) \\ d^{\nabla}\Psi
  \end{pmatrix} \qquad  d^{\tn^{\Cal A}}\begin{pmatrix}\psi\\\Psi
  \end{pmatrix}=\begin{pmatrix} d^{\nabla}\psi-(-1)^k\Alt(\Psi) \\ d^{\nabla}\Psi + (-1)^k 
  i_{\psi}R \end{pmatrix}
  $$
\end{prop}
\begin{proof}
This follows by directly expanding the definition of the covariant exterior derivative
in \eqref{dn-def} using the definitions of the connections in \eqref{conn-def}. 
\end{proof}

There is also a well known general description of the composition of two covariant
exterior derivatives in terms of the curvature of the connection. In particular, this
composition vanishes if and only if the curvature of the connection vanishes. So we
see that on a flat Riemannian manifold, we obtain $d^\nabla\o d^\nabla=0$ on forms
with values in any natural vector bundle and $d^{\nabla^{\Cal A}}\o d^{\nabla^{\Cal
    A}}=0$ on $\Om^*(M,\Cal AM)$, while constant sectional curvature is sufficient to
obtain  $d^{\tn^{\Cal A}}\o d^{\tn^{\Cal A}}=0$ on $\Om^*(M,\Cal AM)$. For future
use we need slightly more information in the case of $d^{\tn^{\Cal A}}$ in the case
of general curved manifolds.

\begin{lemma}\label{lem2.3}
For any element $\al\in\Om^k(M,\Cal AM)$ with $k=0,\dots,n$ the form $d^{\tn^{\Cal
    A}}(d^{\tn^{\Cal A}}\al)$ is a section of the subbundle $\La^kT^*M\otimes\frak
o(TM)\subset \La^kT^*M\otimes\Cal AM$, i.e.\ its top component always vanishes.
\end{lemma}
\begin{proof}
The general results on the composition of two covariant exterior derivatives
mentioned above express $d^{\tn^{\Cal A}}(d^{\tn^{\Cal
    A}}\al)(\xi_1,\dots,\xi_{k+2})$ as a linear combination of terms in which
$\widetilde{R}^{\Cal A}(\xi_i,\xi_j)$ acts on the value of $\al$ on the other vector
fields. Hence the claim readily follows from the description of $\widetilde{R}^{\Cal
  A}$ in Proposition \ref{prop2.2}.  
\end{proof}

\subsection{Harmonic subbundles, splitting operators and BGG operators}\label{2.4}
For each $k=0,\dots,n$, we next define a subbundle $\Cal
H^kM\subset\La^kT^*M\otimes\Cal AM$ as follows. For $k=0$, we put $\Cal
H^0M:=TM\subset\Cal AM$ and for $k=1$, we put
$$
\Cal H^1M:=\Cal S(TM)\subset T^*M\otimes TM\subset T^*M\otimes\Cal AM,
$$
the bundle of endomorphisms of $TM$, which are symmetric with respect to the metric
$g$. For $k\geq 2$, we observe that the operation $\Alt$ defined in \eqref{Alt-def}
actually comes from a vector bundle map $\La^kT^*M\otimes\frak
o(M)\to\La^{k+1}T^*M\otimes TM$, and we define $\Cal H^kM$ to be the kernel of this
bundle map. Hence a section of the bundle $\Cal H^kM$ for $k\geq 2$ simply is a form
$\Psi\in\Om^k(M,\frak o(TM))$ such that $\Alt(\Psi)=0$.

Next, we want to construct, for each $k$, a differential operator $L=L^k:\Ga(\Cal
H^kM)\to\Om^k(M,\Cal AM)$, which admits a simple characterization in each case.
\begin{thm}\label{thm2.4}
(i) Given $\eta\in\frak X(M)$, there is a unique section $\Phi\in\Ga(\frak o(TM))$
  such that $d^{\tn^{\Cal A}}\binom{\eta}{\Phi}$ is a section of $\Cal H^1M\oplus
  (T^*M\otimes\frak o(TM))\subset T^*M\otimes\Cal AM$. Putting
  $L(\eta)=L^0(\eta):=\binom{\eta}{\Phi}$ defines a linear first order differential
  operator
  $$L:\Ga(\Cal H^0M)\to\Ga(\Cal AM)=\Om^0(M,\Cal AM).$$

(ii) Given $\psi\in\Ga(\Cal H^1)\subset\Om^1(M,TM)$, there is a unique
  $\Psi\in\Om^1(M,\frak o(TM))$ such that $d^{\tn^{\Cal A}}\binom{\ps}{\Psi}$ is a
  section of the subbundle $\La^2T^*M\otimes\frak o(TM)\subset\La^2T^*M\otimes\Cal
  AM$. Putting $L(\psi)=L^1(\psi):=\binom{\psi}{\Psi}$ defines a linear first order
  differential operator $$L:\Ga(\Cal H^1M)\to\Om^1(M,\Cal AM).$$

  (iii) For $k\geq 2$ and $\Psi\in\Ga(\Cal H^kM)$, $L(\Psi):=\binom0{\Psi}$ defines a
  tensorial operator $L:\Ga(\Cal H^kM)\to\Om^k(M,\Cal AM)$ and
  $d^{\tn^{\Cal A}}L(\Psi)\in\Ga(\La^{k+1}T^*M\otimes\frak
  o(TM))\subset\Om^{k+1}(M,\Cal AM)$.
\end{thm}
\begin{proof}
(i) Splitting endomorphisms of $TM$ into a symmetric and a skew symmetric part with
  respect to $g$ shows that $T^*M\otimes TM=\frak o(TM)\oplus\Cal S(TM)$. Since this
  is a point-wise property, any element $\psi\in\Om^1(M,TM)$ uniquely splits into the
  sum of a section of $\frak o(TM)$ and a section of $\Cal S(TM)$. From the
  definition in \eqref{conn-def}, it is clear that the top component of $d^{\tn^{\Cal
      A}}\binom{\eta}{\Phi}$ is a section of $\Cal H^1M$ if and only if $\Phi$ is the
  $\frak o(TM)$-component of $\nabla\eta$, and this clearly implies all claims in
  (i).

(ii) Here the key fact is that the bundle map inducing $\Alt$ defines an isomorphism
  $T^*M\otimes\frak o(TM)\to\La^2T^*M\otimes TM$, see e.g.\ Lemma 2 of
  \cite{Arnold-Hu}. (This is induced by the Spencer differential for Riemannian
  structures and the statement that it is an isomorphisms implies existence and
  uniqueness of the Levi-Civita connection.) Hence for
  $\binom{\psi}{\Psi}\in\Om^1(M,\Cal AM)$, the condition that $d^{\tn^{\Cal
      A}}\binom{\psi}{\Psi}$ has vanishing top component is by Proposition
  \ref{prop2.3} equivalent to $\Psi=-\Alt^{-1}(d^\nabla\psi)$. This implies all
  claims in (ii).

  (iii) By Proposition \ref{prop2.3}, $d^{\tn^{\Cal A}}\binom{0}{\Psi}$ has top
  component $(-1)^{k+1}\Alt(\Psi)$ so for $k\geq 2$ this vanishes for
  $\Psi\in\Ga(\Cal H^kM)\subset\Om^k(M,\Cal AM)$.
\end{proof}

Having the operators $L$ at hand, we easily obtain a sequence of differential
operators 
$$D=D^k:\Ga(\Cal H^kM)\to\Ga(\Cal H^{k+1}M).$$ For $k=0$, we know from Proposition
\ref{prop2.4} that for $\eta\in\frak X(M)$, the top component of $\tn^{\Cal
  A}L(\eta)$ is a section of the subbundle $\Cal H^1M\subset T^*M\otimes TM$, so
mapping to this component defines a differential operator $D^0$ as required. For the
remaining cases, we need an additional result.

\begin{prop}\label{prop2.4}
  For $k\geq 1$ and a section $\al\in\Ga(\Cal H^kM)$, $d^{\tn^{\Cal A}}L(\al)$ is a
  section of the subbundle $\Cal H^{k+1}M\subset\La^{k+1}T^*M\otimes\Cal AM$. Thus
  $d^{\tn^{\Cal A}}\o L$ provides an operator $D^k$ as required.
\end{prop}
\begin{proof}
  From parts (ii) and (iii) of Theorem \ref{thm2.4}, we know that for $k\geq 1$,
  $d^{\tn^{\Cal A}}L(\al)$ is a section of the subbundle $\La^{k+1}T^*M\otimes\frak
  o(TM)$. Hence it is of the form $\binom{0}{\Psi}$ and to prove our result, we have
  to show that $\Alt(\Psi)=0$. But by Proposition \ref{prop2.3}, $\Alt(\Psi)$ can be
  computed as $(-1)^{k+1}$ times the top component of
  $$
   d^{\tn^{\Cal A}}\begin{pmatrix} 0\\ \Psi\end{pmatrix}=  d^{\tn^{\Cal
       A}}d^{\tn^{\Cal A}}L(\al), 
   $$
   which vanishes by Lemma \ref{lem2.3}.
 \end{proof}

 \begin{definition}\label{def2.4}
   The operators $L:\Ga(\Cal H^kM)\to\Om^k(M,\Cal AM)$ are called the
   \textit{splitting operators} and the operators
   $D:\Ga(\Cal H^kM)\to\Ga(\Cal H^{k+1}M)$ are called the \textit{BGG operators}
   associated to the deformation connection $\tn^{\Cal A}$. The sequence
   $$
  \Ga(\Cal H^0)\overset{D^0}{\longrightarrow}\Ga(\Cal
  H^1M)\overset{D^1}{\longrightarrow}\cdots
  \overset{D^{n-1}}{\longrightarrow}\Ga(\Cal H^nM)
  $$
  is called the \textit{BGG sequence} associated to the deformation connection
  $\tn^{\Cal A}$. 
 \end{definition}

 \subsection{Relation to the twisted de Rham sequence}\label{2.5}
 We next discuss the relation of the BGG sequence to the twisted de Rham sequence
 $(\Om^*(M,\Cal AM),d^{\tn^{\Cal A}})$. The main tool to do this are the splitting
 operators $L$ constructed in Theorem \ref{thm2.4}.

 \begin{thm}\label{thm2.5}
   For any $k=0,\dots, n-1$, we get a commutative diagram
   $$\begin{CD}
     \Om^k(M,\Cal AM) @>d^{\tn^{\Cal A}}>> \Om^{k+1}(M,\Cal AM)\\
     @AL^kAA   @AL^{k+1}AA \\
     \Ga(\Cal H^kM) @>D^k>>  \Ga(\Cal H^{k+1}M)
   \end{CD}
   $$ If the metric $g$ has constant sectional curvature, then $d^{\tn^{\Cal
       A}}\o d^{\tn^{\Cal A}}=0$ and $D^{k+1}\o D^k=0$, so both sequences are
   complexes, and the operators $L$ define a chain map that induces an isomorphism in
   cohomology.
 \end{thm}
 \begin{proof}
For $k\geq 1$, commutativity of the diagram follows straight from the construction:
For $\al\in\Ga(\Cal H^kM)$, we have shown in Proposition \ref{prop2.4} that
$d^{\tn^{\Cal A}}L(\al)$ always is a section of the subbundle $\Cal H^{k+1}\subset
\La^{k+1}T^*M\otimes\Cal AM$. We have defined this to be $D^k(\al)$, while by
definition $L^{k+1}$ is just the inclusion of the sections of the subbundle $\Cal
H^{k+1}$ into $\Om^{k+1}(M,\Cal A^*M)$, so the claim follows.

For $k=0$ and $\eta\in\frak X(M)$, we have defined $D(\eta)$ as the top component of
$\tn^{\Cal A}L(\eta)=d^{\tn^{\Cal A}}L(\eta)$, which always lies in $\Ga(\Cal
H^1M)$. Hence $\tn^{\Cal A}L(\eta)\in\Om^1(M,\Cal A)$ has the form
$\binom{\eta}{\Phi}$ for some $\Phi\in\Ga(\frak o(TM))$, and by Lemma \ref{lem2.3} we
know that $d^{\tn^{\Cal A}}(\tn^{\Cal A}L(\eta))$ has vanishing top component. The
uniqueness in part (ii) of Theorem \ref{thm2.4} thus implies that $\tn^{\Cal
  A}L(\eta)=L(D(\eta))$.

If $g$ has constant sectional curvature, we have shown in Proposition \ref{prop2.2}
that the connection $\tn^{\Cal A}$ is flat and then it is well-known that
$(\Om^*(M,\Cal AM),d^{\tn^{\Cal A}})$ is a complex. Commutativity of the diagrams
then implies that
$$
L^{k+1}\o D^{k+1}\o D^k=d^{\tn^{\Cal A}}\o L^k\o D^k=d^{\tn^{\Cal A}}\o d^{\tn^{\Cal
    A}}\o L^k=0
$$
and by construction the operator $L^{k+1}$ is injective. Hence $(\Ga(\Cal H^*M),D)$ is a
complex and then commutativity of the diagrams exactly says that $L$ is a chain map
and hence induces a map in cohomology.

To show that this induced map is an isomorphism, we start in degree $0$. For a
section $\binom{\eta}{\Phi}$ of $\Cal AM$ such that $\tn^{\Cal
  A}\binom{\eta}{\Phi}=0$, uniqueness in part (i) of Theorem \ref{thm2.5} implies
that $\binom{\eta}{\Phi}=L(\eta)$. But $\tn^{\Cal A}\o L=L\o D$ shows that $\tn^{\Cal
  A}L(\eta)=0$ is equivalent to $D(\eta)=0$ which proves bijectivity of the induced
map in cohomology in degree $0$.

For $\binom{\psi}{\Psi}\in\Om^1(M,\Cal AM)$ we can split $\psi\in\Om^1(M,TM)$ into a
symmetric part $\Cal S(\psi)$ and an alternating part $\Phi$. But then Proposition
\ref{prop2.3} shows that $\binom{\psi}{\Psi}+\tn^{\Cal A}\binom{0}{\Phi}=\binom{\Cal
  S(\psi)}{\tilde\Psi}$ for some $\tilde\Psi\in\Om^1(M,\frak o(TM))$. If
$\binom{\psi}{\Psi}$ is a cocycle, then also $d^{\tn^{\Cal A}}\binom{\Cal
  S(\psi)}{\tilde\Psi}=0$. As in degree zero, we conclude that $\binom{\Cal
  S(\psi)}{\tilde\Psi}=L(S(\psi))$ and that $d^{\tn^{\Cal A}}L(\tau)=0$ is equivalent
to $D(\tau)=0$ for any $\tau\in\Ga(\Cal H^1M)$. Hence bijectivity of the induced map
in cohomology in degree one follows. 

In degrees $\geq 2$, things are even easier, since Lemma 2 of \cite{Arnold-Hu} shows
that any section of $\La^kT^*M\otimes TM$ can be written as $\Alt(\Psi)$ for some
$\Psi\in\Ga(\La^{k-1}T^*M\otimes\frak o(TM))$. Using Proposition \ref{prop2.3}, we
conclude that for any $\al\in\Om^k(M,\Cal AM)$, there is a $\be\in\Om^{k-1}(M,\Cal
AM)$ such that $\al+d^{\tn^{\Cal A}}\be$ has vanishing top component. But the
form $\al=\binom{0}{\Psi}\in\Om^k(M,\Cal AM)$ is closed for $d^{\tn^{\Cal A}}$ if and
only if $\Alt(\Psi)=0$, so $\Psi\in\Ga(\Cal H^kM)$ and $\al=L(\Psi)$ and
$d^{\tn^{\Cal A}}\al=D(\Psi)$.
 \end{proof}

 \begin{remark}\label{rem2.4}
   One can run an analogous construction based on the adjoint tractor connection
   $\nabla^{\Cal A}$ rather than on the deformation connection $\tn^{\Cal A}$, but
   things become a bit more complicated in this case. In the setting of Theorem
   \ref{thm2.4}, nothing changes if one replaces $\tn^{\Cal A}$ by $\nabla^{\Cal A}$,
   so one obtains the same splitting operators $L$ from $\nabla^{\Cal A}$. Also, the
   construction of $D^0$ does not depend on whether one uses $d^{\nabla^{\Cal A}}$ or
   $d^{\tn^{\Cal A}}$, so the first BGG operator remains unchanged, too. Both these
   facts are parallel to what happens for parabolic geometries, compare to Section
   3.5 of \cite{deformations}. In degrees $k\geq 2$, one can follow the same line of
   argument as in the proof of Proposition \ref{prop2.4} to show that for
   $\al\in\Ga(\Cal H^k)$, $d^{\nabla^{\Cal A}}L(\al)$ is a section of the subbundle
   $\Cal H^{k+1}$, since for $k\geq 2$, $L(\al)$ itself has vanishing top component.

   For $k=1$, things really get more complicated. For $\psi\in\Ga(\Cal H^1M)$, it is
   of course still true that $d^{\nabla^{\Cal A}}L(\psi)$ is a section of the
   subbundle $\La^2T^*M\otimes\frak o(TM)$. One can compute $\Alt(d^{\nabla^{\Cal
       A}}L(\psi))$ using the same method as in the proof of Proposition
   \ref{prop2.4}, but here this produces a curvature term that is non-zero in
   general. One can still define an associated BGG operator by taking the component
   of $d^{\nabla^{\Cal A}}L(\psi)$ in $\Ga(\Cal H^2M)$ and this component can be
   computed explicitly using the curvature term derived before.
 \end{remark}

 \subsection{Explicit formulae}\label{2.6}
 It is easy to compute explicit formulae for both the splitting operators and the BGG
 operators. For the computation in degree zero, the main observation is that the
 splitting of $\psi\in\Om^1(M,TM)$ into symmetric and skew symmetric parts can be
 written in terms of a local orthonormal frame $\{s_i\}$ for $(M,g)$ as
 $$
 \ze\mapsto \tfrac{1}{2}\left(\ps(\zeta)\pm \textstyle\sum_ig(\ps(s_i),\ze)s_i\right).
 $$
 By the proof of Theorem \ref{thm2.4}, this immediately implies that writing
 $L(\eta)=\binom{\eta}{\Phi}$ we get for any $\ze\in\frak X(M)$:
 \begin{gather}\label{L0}
   \Phi(\zeta)=\tfrac12\left(\nabla_\ze \eta-
     \textstyle\sum_ig(\nabla_{s_i}\eta,\ze)s_i\right)\\
   \label{D0}
   D^0(\eta)(\zeta)=\tfrac12\left(\nabla_\ze \eta+
     \textstyle\sum_ig(\nabla_{s_i}\eta,\ze)s_i\right).
   \end{gather} 
   In (abstract) index notation, this reads as
   $$
L(\eta)=\begin{pmatrix}  \eta^a \\ \frac12(\nabla_b\eta^c-g^{ci}g_{bj}\nabla_i\eta^j )
\end{pmatrix} \qquad
(D^0(\eta))_a^b=\tfrac12(\nabla_a\eta^b+g^{bi}g_{aj}\nabla_i\eta^j), 
$$
so we obtain the Killing operator as the first BGG operator. In particular Theorem
\ref{thm2.4} implies that for a metric of constant sectional curvature, the BGG
sequence is a fine resolution of the sheaf of local Killing vector fields. Thus we
have obtained alternative proofs the results of \cite{Calabi} and \cite{BBL}. 

\medskip

In degree one, we have to start from $\ps\in\Ga(\Cal S(TM))$, i.e.~a
$\binom11$-tensor field such that the map $(\xi_1,\xi_2)\mapsto g(\psi(\xi_1),\xi_2)$
is symmetric in $\xi_1,\xi_2$. Observe that this implies that also $\nabla_\xi\ps$ is
a section of $\Cal S(TM)$ for any $\xi\in\frak X(M)$. This implies that defining a
$\binom03$-tensor field $A^\psi$ by 
\begin{equation}\label{Apsi-def}
A^\psi (\xi_1,\xi_2,\xi_3):=g((\nabla_{\xi_1}\ps)(\xi_2),\xi_3)
\end{equation}
for $\xi_i\in\frak X(M)$, the result is symmetric in $\xi_2$ and $\xi_3$. Now we
define $\Ps\in\Om^1(M,L(TM,TM))$ by requiring that
\begin{equation}\label{Psi-def}
g(\Ps(\xi_1)(\xi_2),\xi_3)=A^\psi(\xi_2,\xi_1,\xi_3)-A^\psi(\xi_3,\xi_1,\xi_2).
\end{equation}
This evidently changes sign if we exchange $\xi_2$ and $\xi_3$, which means that
$\Psi(\xi_1)\in\frak o(TM)$ for each $\xi_1$, so $\Psi\in\Om^1(M,\frak o(TM))$. On
the other hand, using \eqref{Psi-def} to compute
$g(\Ps(\xi_1)(\xi_2)-\Ps(\xi_2)(\xi_1),\xi_3)$ then using symmetry of $A^\psi$ in the last
two entries, we obtain $A^\psi(\xi_2,\xi_1,\xi_3)-A^\psi(\xi_1,\xi_2,\xi_3)$. This
shows that 
$$
\Ps(\xi_1)(\xi_2)-\Ps(\xi_2)(\xi_1)=(\nabla_{\xi_2}\psi)(\xi_1)-(\nabla_{\xi_1}\psi)(\xi_2),
$$
which coincides with $-d^\nabla\ps(\xi_1,\xi_2)$ by torsion-freeness of
$\nabla$. Hence $\Alt(\Psi)=-d^{\nabla}\psi$, which shows that
$L(\psi)=\binom{\psi}{\Psi}$, or in abstract index notation
$$
L(\psi)=\begin{pmatrix} \psi_a^b \\ \nabla_d\psi_c^e-g_{di}g^{je}\nabla_j\psi_c^i\end{pmatrix}. 
$$

Here the convention for the lower row is that $c$ denotes the $1$-form index, while
$d$ and $e$ are the $\frak o(TM)$-indices.  From this, we can easily compute the
value $D(\psi)\in\Om^2(M,\frak o(TM))$ of BGG operator. By Proposition \ref{prop2.3},
this is given by $d^{\nabla}\Psi-i_{\ps}R$ and we obtain
$$
D(\psi)(\xi_1,\xi_2)=\nabla_{\xi_1}\Psi(\xi_2)-\nabla_{\xi_2}\Psi(\xi_1)-\Psi([\xi_1,\xi_2])-R(\psi(\xi_1),\xi_2)+R(\psi(\xi_2),\xi_1). 
$$
By torsion-freeness of $\nabla$, the first three terms in the right hand side can be
equivalently rewritten as
$(\nabla_{\xi_1}\Psi)(\xi_2)-(\nabla_{\xi_2}\Psi)(\xi_1)$. In the latter version, we
can easily rewrite explicitly in terms of $\psi$ only. Letting $a$ and $b$ denote the
$2$-form indices we obtain the following expression for $D(\psi)_{ab}{}^c_d$:
$$
\nabla_a\nabla_d\psi_b^c-\nabla_b\nabla_d\psi_a^c
-g_{di}g^{cj}(\nabla_a\nabla_j\psi^i_b-\nabla_b\nabla_j\psi^i_a)
-\psi^i_aR_{ib}{}^c{}_d+\psi^i_bR_{ia}{}^c{}_d.
$$

After an appropriate interpretation of the double covariant derivatives used in
\cite{BBL} this is exactly twice the formula obtained there, which is attributed to
M.\ Berger's work \cite{Berger}. The factor $2$ will be explained in Proposition
\ref{prop3.5}, see also Theorem \ref{thm3.6}. 

\medskip

In higher degrees, it is much easier to obtain explicit formulae. For
$k\geq 2$, a section of $\Cal H^k$ is $\Psi\in\Om^k(M,\frak o(TM))$
such that $\Alt(\Psi)=0$. Moreover $L(\Psi)=\binom{0}{\Psi}$ and by Propositions
\ref{prop2.4} and \ref{prop2.3}, we get $D(\Psi)=d^{\nabla}\Psi\in\Ga(\Cal H^{k+1})$,
so the explicit formula can be read off directly from \eqref{dn-def}.

\section{Interpretation via Cartan geometries}\label{3}
We start by reviewing the equivalent description of Riemannian manifolds as Cartan
geometries that satisfy a normalization condition, see \cite{Sharpe} and Section 1.5
of \cite{book} for details and generalities on Cartan geometries. The main advantage
of this picture to be exploited below is that linearized deformations of Cartan
geometries admit a simple description via a twisted de Rham sequence.

\subsection{The Cartan description of Riemannian metrics}\label{3.1}
Let $(M,g)$ be a Riemannian manifold of dimension $n\geq 2$. Then one can form the
\textit{orthonormal frame bundle} $p:\Cal OM\to M$ whose fiber over $x\in M$ consists
of all linear isomorphisms $u:\Bbb R^n\to T_xM$ such that $g_x(u(v),u(w))=\langle
v,w\rangle$. Here $\langle\ ,\ \rangle$ denotes the standard inner product on $\Bbb
R^n$. This is a locally trivial principal fiber bundle with structure group
$O(n)$. It comes with the principal right action $r:\Cal OM\x O(n)\to\Cal OM$
defined by $r(u,A)=u\o A$, so this is free and transitive on each fiber.

On $\Cal OM$, there is a canonical one-form $\th\in\Om^1(\Cal OM,\Bbb R^n)$, called
the \textit{soldering form}. This sends $\xi\in T_u\Cal OM$ to
$\th(u)(\xi):=u^{-1}(T_up(\xi))\in\Bbb R^n$, which readily implies that
$\ker(\th(u))=\ker(T_up)$, so $\th$ is \textit{strictly horizontal}. On the other
hand, taking $A\in O(n)$ one obtains the map $r^A:\Cal OM\to\Cal OM$ defined by
$r^A(u):=r(u,A)$, which satisfies $p\o r^A=r^A$. Differentiating this and using the
definition of $r$, one readily concludes that $(r^A)^*\th=A^{-1}\o\th$, so $\th$ is
$O(n)$-equivariant.

By construction, the tangent bundle $TM$ is the associated bundle $\Cal
OM\x_{O(n)}\Bbb R^n$ corresponding to the standard representation of $O(n)$ on $\Bbb
R^n$. Hence any principal connection on $\Cal OM$ induces a linear connection on $TM$
and is uniquely determined by this induced connection. It is easy to see that a
linear connection $\hat\nabla$ on $TM$ is induced from a principal connection on
$\Cal OM$ if and only if it is metric for $g$, i.e.\ $\hat\nabla g=0$ for the induced
connection on $S^2T^*M$ or
$$
\xi\cdot g(\eta,\zeta)=g(\nabla_\xi\eta,\zeta)+g(\eta,\nabla_\xi\zeta)
$$
for $\xi,\eta,\zeta\in\frak X(M)$. The standard description of a principal
connection on $\Cal OM$ is as a one-form $\gamma\in\Om^1(\Cal OM,\frak o(n))$, which
is $O(n)$-equivariant, i.e.~$(r^A)^*\gamma=\Ad(A^{-1})\o\gamma$ for the adjoint
representation $\Ad$ of $O(n)$ on $\frak o(n)$. In addition, one has to require that
$\gamma$ reproduces the generators of fundamental vector fields, so
$\gamma(\ze_X)=X$, where $\ze_X(u)=\tfrac{d}{dt}|_{t=0}r^{\exp(tX)}(u)$.

The Cartan description provides a neat way to combine $\th$ and $\gamma$ into a
single object. Consider the group $\Euc(n)$ of Euclidean motions on $\Bbb R^n$ and
let $\euc(n)$ be its Lie algebra. We can realize $\Euc(n):=O(n)\x\Bbb R^n$ with
multiplication defined by
$$
(A,v)\cdot (B,w):=(AB,Aw+v),
$$
so this is a semi-direct product, and mapping $w$ to the $\Bbb R^n$-component of
$(A,v)\cdot (0,w)$ realizes the usual action of $\Euc(n)$ on $\Bbb
R^n$. Correspondingly, $\euc(n)=\frak o(n)\x\Bbb R^n$ with the Lie bracket
\begin{equation}\label{euc-bracket}
[(X,v),(Y,w)]=([X,Y],Xw-Yv). 
\end{equation}

Given forms $\th$ and $\gamma$ as above, we define 
$\om:=\gamma\oplus\th\in\Om^1(\Cal OM,\euc(n))$. Now by construction $H:=O(n)$ is a
Lie subgroup of $G:=\Euc(n)$ and so we can restrict the adjoint representation $\Ad$
of $G$ to obtain a representation of $H$ on $\euc(n)$. This extends the adjoint
representation of $H$ on $\frak h:=\frak o(n)$ to a representation on $\frak
g:=\euc(n)$, so we again denote it by $\Ad$. Explicitly $\Ad(A)(X,v)=(\Ad(A)(X),Av)$,
so as representations of $O(n)$, we get $\euc(n)\cong \frak o(n)\oplus\Bbb
R^n$. Hence the equivariancy properties of $\th$ and $\gamma$ are equivalent to
\begin{equation}\label{om-equiv}
  (r^A)^*\om=\Ad(A^{-1})\o\om \qquad \forall A\in G. 
\end{equation}
Likewise, $\frak h$ is a Lie subalgebra of $\frak g$, so the conditions that
$\th(\ze_X)=0$ and that $\gamma$ reproduces the generators of fundamental vector fields
can be combined to
\begin{equation}\label{om-fund}
  \om(\ze_X)=X\qquad \forall X\in\frak h\subset\frak g.
\end{equation}
Finally, since
the kernel of $\th(u)$ consists of vertical vectors only, we conclude that
\begin{equation}\label{om-iso}
  \om(u):T_u\Cal OM\to\euc(n)\text{\ is a linear isomorphism\ }\forall u\in\Cal OM,
\end{equation}
since it has to be injective.

\eqref{om-equiv}--\eqref{om-iso} are the defining properties of a \textit{Cartan
  connection} of type $(G,H)$ on a principal $H$-bundle. The usual interpretation of
this is that $(\Cal OM,\om)$ is a ``curved analog'' of the homogeneous space $G/H$,
which is Euclidean $n$-space. For this \textit{homogeneous model} the corresponding
description is given by the principal $H$-bundle $G\to G/H$ and the left
Maurer-Cartan form. Indeed, such a \textit{Cartan geometry} $(\Cal OM,\om)$ on $M$ is
equivalent to a Riemannian metric $g$ on $M$ plus a linear connection on $TM$ such
that $\nabla g=0$, see Chapter 6 of \cite{Sharpe} for details.

\begin{thm}\label{thm3.1}
Let $M$ be a smooth manifold of dimension $n$, $\pi:P\to M$ a principal $O(n)$-bundle
and $\om\in\Om^1(P,\euc(n))$ a Cartan connection of type $(\Euc(n),O(n))$. Then $\om$
induces a Riemannian metric $g$ on $M$ and an isomorphism $\Ph:P\to \Cal OM$ of
principal fiber bundles such that $(\Ph^{-1})^*\om=\th\oplus\gamma$ for the soldering
form $\th$ and a principal connection form $\ga$.
\end{thm}
\begin{proof}[Sketch of proof] Given $x\in M$, choose $y\in P$ with $\pi(y)=x$ and
  observe that the linear isomorphism $\om(y):T_yP\to\frak g$ maps $\ker(T_y\pi)$
  isomorphically onto $\frak h$. Thus it induces a linear isomorphism
  $T_yP/\ker(T_y\pi)\to\frak g/\frak h=\Bbb R^n$, and via $T_y\pi$, the left hand
  space is isomorphic to $T_xM$.  We denote the inverse of this by $\Ph(y):\Bbb
  R^n\to T_xM$ and observe that for a different choice $\tilde y\in P$, we obtain
  $\Ph(\tilde y)=\Ph(y)\o A$ for some $A\in O(n)$ by equivariancy of $\om$. Using
  $\Ph(y)$ to move over $\langle\ ,\ \rangle$ to an inner product on $T_xM$, the
  result is independent of the choice of $y$ and one proves that these inner products
  fit together to define a Riemannian metric $g$ on $M$. But then by definition
  $\Ph(y)\in\Cal OM$ and $\Ph:P\to\Cal OM$ defines an isomorphism of principal
  bundles and by construction, the $\Bbb R^n$-component of $(\Ph^{-1})^*\om$
  reproduces the soldering form $\th$. The defining properties of $\om$ then easily
  imply that the $\frak o(n)$-component of $(\Ph^{-1})^*\om$ is a principal
  connection form on $\Cal OM$.
\end{proof}

\subsection{The Levi-Civita connection}\label{3.2}
To complete the Cartan description of Riemannian manifolds, we have to implement
existence and uniqueness of the Levi-Civita connection. This can be nicely proved in
the setting of Cartan geometries, see \cite{Sharpe}, but here we just use and
interpret the well known facts from Riemannian geometry. A linear connection on $TM$
has a torsion $T\in \Om^2(M,TM)$ and a curvature $R\in\Om^2(M,L(TM,TM))$ and for
metric connections, the curvature has values in $\frak o(TM)\subset L(TM,TM)$. The
torsion and the curvature admit a nice description in the Cartan picture, which needs
one more well known result. Recall that for a principal $H$-bundle $\pi:P\to M$ and
any representation $\rho$ of $O(n)$ on a vector space $V$, one can form the
associated bundle $P\x_HV$. It is well known that then sections $s$ of $P\x_HV$ are
in bijective correspondence with smooth functions $f:P\to V$ which are
$O(n)$-equivariant in the sense that $f\o r^A=\rho_{A^{-1}}\o f$, where $\rho_B$
denotes the action of $B$ on $V$. More generally, one relates differential forms with
values in $P\x_H V$, i.e.\ sections of the bundle $\La^kT^*M\otimes(P\x_HV)$ to
$V$-valued differential forms on $P$. Recall that for each vector field $\xi\in\frak
X(M)$, there exist lifts $\widetilde{\xi}\in\frak X(P)$,
i.e.\ $T_y\pi(\widetilde{\xi}(y))=\xi(\pi(y))$ for all $y\in P$, which are
$H$-invariant, i.e.\ satisfy $(r^A)^*\widetilde{\xi}=\widetilde{\xi}$ for all $A\in
H$. For reference we formulate the correspondence as a proposition, see Section 19.14
in \cite{Michor:topics} for a proof.
\begin{prop}\label{prop3.2}
  The space $\Om^k(M,P\x_H V)$ is in natural bijective correspondence with the
  subspace $\Om^k_{\hor}(P,V)^H$ consisting of those $k$-forms $\ph$ on $P$, which are
  $H$-equivariant in the sense that $(r^A)^*\ph=\rho_{A^{-1}}\o\ph$ and horizontal in
  the sense that they vanish upon insertions of one vector in $\ker(Tp)$.

  Explicitly, the form $\underline{\ph}\in \Om^k(M,P\x_H V)$ corresponding to $\ph\in
  \Om^k_{\hor}\Om^k(P,V)^H$ maps vector fields $\xi_1,\dots,\xi_k\in\frak X(M)$ to
  the section of $P\x_HV$ corresponding to the $H$-equivariant function
  $\ph(\tilde\xi_1,\dots,\tilde\xi_k)$. Here $\tilde\xi_i\in\frak X(P)$ is an
  $H$-invariant lift of $\xi_i$ for each $i$.
\end{prop}

Now given a Cartan connection $\om\in\Om^1(P,\euc(n))$ of type $(G,H)$, we can define
a $\euc(n)$-valued two-form $K\in\Om^2(P,\euc(n))$ by
\begin{equation}\label{K-def}
K(\xi,\eta)=d\om(\xi,\eta)+[\om(\xi),\om(\eta)] 
\end{equation}
for $\xi,\eta\in\frak X(M)$, where the bracket is in $\euc(n)$. The defining
properties of a Cartan connection easily imply (see \cite{Sharpe}) that
$K\in\Om^2_{\hor}(P,\frak g)^H$ so that it corresponds to
$\ka\in\Om^2(M,P\x_H\euc(n))$. In general, this \textit{Cartan curvature} is a
complicated object but in the Riemannian setting things become easy. As we have seen
above, $\euc(n)=\frak o(n)\oplus\Bbb R^n$ as a representation of $O(n)$, so passing
to associated bundles, we see that $P\x_H\euc(n)=\frak o(TM)\oplus TM$. Hence $\ka$
decomposes into $R\in\Om^2(M,\frak o(TM))$ and $T\in\Om^2(M,TM)$ and these turn out
to be the curvature and the torsion of the metric connection on $TM$ encoded by
$\om$. Existence and uniqueness of the Levi-Civita connection then imply

\begin{thm}\label{thm3.2}
For a Riemannian manifold $(M,g)$ of dimension $n$, there is a unique Cartan
connection $\om$ on $\Cal OM$ such that the form $K$ defined in \eqref{K-def} has
values in $\frak o(n)\subset\euc(n)$.

If $P\to M$ is a principal $O(n)$-bundle endowed with a Cartan connection
$\widehat{\om}$ that induces $g$ and satisfies the same condition on its (Cartan)
curvature, there is an isomorphism $\Ph:P\to\Cal OM$ of principal bundles such that
$\Ph^*\om=\widehat{\om}$.
\end{thm}

\subsection{The connections $\nabla^{\Cal A}$ and $\tn^{\Cal A}$ in the Cartan picture}\label{3.3}
Consider the Cartan geometry $(\Cal OM,\om)$ associated to a Riemannian manifold
$(M,g)$. The adjoint tractor bundle is then defined as $\Cal AM:=\Cal
OM\x_{O(n)}\euc(n)$ so sections of $\Cal AM$ are in bijective correspondence with
$O(n)$-equivariant smooth functions $f:\Cal OM\to\euc(n)$. Since $\om$ trivializes
$T\Cal OM$, smooth functions $f:\Cal OM\to\euc(n)$ are in bijective correspondence
with vector fields $\tilde\xi\in\frak X(\Cal OM)$ via
$\tilde\xi(u)=(\om(u))^{-1}(f(u))$. In this picture, equivariancy of $f$ is
equivalent to the fact that $\tilde\xi$ is $O(n)$-invariant, i.e.\ that
$(r^A)^*\tilde\xi=\tilde\xi$ for any $A\in O(n)$. Hence $\Ga(\Cal AM)\cong \frak
X(\Cal OM)^{O(n)}$, which provides an alternative geometric interpretation of
sections of $\Cal AM$. From Section \ref{3.2}, we know that $\Cal AM\cong \frak
o(TM)\oplus TM$ and the resulting projection $\Ga(\Cal AM)\to\Ga(TM)=\frak X(M)$ has
a neat interpretation in this alternative picture. An $O(n)$-invariant vector field
is automatically projectable and thus projects to a vector field on $M$ which is
exactly is the projection of the corresponding section of $\Cal AM$. Hence we can
view sections of $\Cal AM$ as $O(n)$-invariant lifts of vector fields on $M$ and the
construction readily implies that the sections of $\Cal AM$ with trivial component in
$\Ga(\frak o(TM))$ are exactly the horizontal lifts with respect to the Levi-Civita
connection.

It is well known that differentiating equivariant functions with respect to these
horizontal lifts again leads to equivariant functions, and this describes the
Levi-Civita connection on any natural vector bundle. However, the Cartan connection
$\om$ gives rise to another operation on $\Cal AM$. Given $f\in C^\infty(\Cal
OM,\euc(n))^{O(n)}$ we consider $d^\om f\in\Om^1(\Cal OM,\euc(n))$ defined by
  \begin{equation}\label{dom-def}
    d^\om f=df+[\om,f],
  \end{equation}
  where $[\om,f](\tilde\eta)(u)=[\om(\tilde\eta)(u),f(u)]$ for any
  $\tilde\eta\in\frak X(\Cal OM)$ with the bracket in $\euc(n)$.

On the other hand, for $\tilde\xi\in\frak X(\Cal OM)^{O(n)}$, any local flow of
$\tilde\xi$ is $O(n)$-equivariant, so we can naturally view $\tilde\xi$ as an
infinitesimal automorphism of the principal bundle $\Cal OM$. It is natural to ask
what it means that these flows preserves $\om$ and of course this is equivalent to
vanishing of the Lie derivative $\Cal L_{\tilde\xi}\om$. This operation is closely
related to but surprisingly not equal to $d^{\om}$:

\begin{thm}\label{thm3.3}
  Let $f\in C^{\infty}(\Cal OM,\euc(n))^{O(n)}$ be an equivariant smooth function.
  
  (1) The one-form $d^\om f$ lies in $\Om^1_{\hor}(\Cal OM,\euc(n)))$, so $d^\om$
  induces an operation $\Ga(\Cal AM)\to\Om^1(M,\Cal AM)$. This coincides with the
  adjoint tractor connection $\nabla^{\Cal A}$ from Definition \ref{def2.2}.

  (2) If $f$ corresponds to $\tilde\xi\in\frak X(\Cal OM)^{O(n)}$ then
  $$
  \Cal L_{\tilde\xi}\om = d^\om f+K(\tilde\xi,\_)\in \Om^1(\Cal OM,\euc(n)).  
  $$ In particular, this is horizontal and $O(n)$-equivariant. Thus also
  $\tilde\xi\mapsto \Cal L_{\tilde\xi}\om$ induces an operation $\Ga(\Cal
  AM)\to\Om^1(M,\Cal AM)$ and this coincides with the deformation connection
  $\tn^{\Cal A}$ from Definition \ref{def2.2}.
\end{thm}
\begin{proof}
  (1) Since $d$ commutes with pullbacks, we get
  $$
  (r^A)^*df=d((r^A)^*f)=d(f\o r^A).
  $$
  By equivariancy of $f$, we have $f\o r^A=\Ad(A^{-1})\o f$ and since
  $\Ad(A^{-1})$ only acts on the values of $f$, the definition of the exterior
  derivative readily implies that $d(\Ad(A^{-1})\o f)=\Ad(A^{-1})\o df$. Thus $f\in
  C^{\infty}(\Cal OM,\frak g)^{O(n)}$ implies $df\in\Om^1(\Cal OM,\frak g)^{O(n)}$.

  The definition of the pullback implies that $(r^A)^*([\om,f])=[(r^A)^*\om,f\o
    r^A]$. Equivariancy of $\om$ and $f$ imply that this coincides with
  $$
  [\Ad(A^{-1})\o \om,\Ad(A^{-1})\o f]=\Ad(A)^{-1}\o [\om,f],
  $$ where we have used that $\Ad(A^{-1})$ is a homomorphism of Lie algebras in the
  last step. Together with the above, this shows that $d^\om f$ is
  $O(n)$-equivariant. To show that $d^\om f$ is horizontal, it suffices to show that
  for any $X\in\frak o(n)$, $d^{\om} f$ vanishes on the fundamental vector field
  $\ze_X$. By definition of the fundamental vector field, we can compute
  $df(\ze_X)(u)=(\ze_X\cdot f)(u)$ as
  \begin{align*}
  \tfrac{d}{dt}|_{t=0}f(u\o\exp(tX))=&
  \tfrac{d}{dt}|_{t=0}(\Ad(\exp(-tX))(f(u))\\ =&-[X,f(u)]=-[\om(\ze_X)(u),f(u)], 
  \end{align*}
  which implies that $d^\om f(\ze_X)=0$.

  Knowing that $d^{\om}f$ is horizontal and equivariant, we can describe the induced
  operation $\Ga(\Cal AM)\to \Om^1(M,\Cal AM)$ by evaluating it on any lift of a
  vector field $\xi\in\frak X(M)$. We use the horizontal lift $\xi^{\hor}$ with
  respect to the Levi-Civita connection, which by definition has the property that
  $\om(\xi^{hor})$ has values in $\Bbb R^n\subset\euc(n)$ and is the equivariant
  function corresponding to $\xi$. It is well known that the equivariant function
  $\xi^{\hor}\cdot f$ corresponds to the (Levi-Civita) covariant derivative of the
  section $s\in\Ga(\Cal AM)$ corresponding to $f$. Writing $s=\binom{\eta}{\Phi}$ as
  before, this hence produces $\binom{\nabla_{\xi}\eta}{\nabla_\xi\Phi}$. On the
  other hand, formula \eqref{euc-bracket} with $X=0$ immediately implies that
  $[\om(\xi^\hor),f]$ corresponds to $\binom{-\Ph(\xi)}{0}$ and this completes the
  proof of (i).

(2) Since $K\in\Om^2(\Cal OM,\euc(n))$ is horizontal and $O(n)$-equivariant, one
  immediately concludes that for $\tilde\xi\in\frak X(\Cal OM)^{O(n)}$ we get
  $K(\tilde\xi,\_ )\in\Om^1_{\hor}(\Cal OM,\frak g)^{O(n)}$. Hence by part (1), the
  right hand side of the claimed equation is horizontal and equivariant. Hence it
  induces an operation $\Ga(\Cal AM)\to \Om^1(M,\Cal AM)$ and part (1) and the
  relation of the form $\ka\in\Om^2(M,\Cal AM)$ induced by $K$ to the Riemann
  curvature $R$ show that this equals $\tn^{\Cal A}$.

  Thus it remains to show the relation to $\Cal L_{\tilde\xi}\om$. For a
  vector field $\tilde\eta\in\frak X(\Cal OM)$, we by definition get
  \begin{align*}
(\Cal L_{\tilde\xi}\om)(\tilde\eta)&=\tilde\xi\cdot\om(\tilde\eta)-
    \om([\tilde\xi,\tilde\eta])=d\om(\tilde\xi,\tilde\eta)+
    \tilde\eta\cdot\om(\tilde\xi)\\ &=K(\tilde\xi,\tilde\eta)-
                      [\om(\tilde\xi),\om(\tilde\eta)]+\tilde\eta\cdot\om(\tilde\xi).
  \end{align*}
  But writing $f:=\om(\tilde\xi)$ and using skew symmetry of the Lie bracket, the
  last two terms just give $d^\om f(\tilde\eta)$.
\end{proof}

\subsection{The Cartan deformation sequence}\label{3.4}
Both $d^\om$ and the modification by $K$ can be extended to operators acting on
horizontal equivariant forms of higher degree. This can be verified directly, but in
our situation it is easier to just deduce this from the covariant exterior
derivatives $d^{\nabla^{\Cal A}}$ and $d^{\tn^{\Cal A}}$. Indeed, for
$\ph\in\Om^k_{\hor}(\Cal OM,\frak g)^{O(n)}$ one has to consider
$d^{\om}\ph:=d\ph+[\om,\ph]$ and $\widetilde{d}^{\om}\ph:=d^{\om}\ph+(-1)^ki_\ph K$,
where for $\xi_i\in\frak X(\Cal OM)$ one puts
\begin{gather}\label{br-def}
[\om,\ph](\xi_0,\dots,\xi_k):=\textstyle\sum_{i=0}^k(-1)^i
[\om(\xi_i),\ph(\xi_0,\dots,\widehat{\xi_i},\dots,\xi_k)] \\
\label{iK-def}  (i_\ph K)(\xi_0,\dots,\xi_k):=\textstyle\sum_{i=0}^k(-1)^{k-i}
  K(\om^{-1}(\ph(\xi_0,\dots,\widehat{\xi_i},\dots,\xi_k)),\xi_i), 
\end{gather}
with the hats denoting omission.

The advantage of the ``upstairs'' sequence $(\Om^*_{\hor}(\Cal OM,\frak
g)^{O(n)},\widetilde{d}^\om)$ is that the first steps in this sequences have a direct
interpretation as a deformation sequence. The first step of this was already
discussed in Section \ref{3.3}: Sections of $\Cal AM$ can be naturally viewed as
$O(n)$-invariant vector fields on $\Cal OM$ and of course any local flow of such a
vector is $O(n)$-equivariant and hence a local automorphism of the principal bundle
$\Cal OM$. Conversely, for an open interval $I\subset\Bbb R$ with $0\in I$ and a
smooth family $\{\Ps_t:t\in I\}$ of such automorphisms, $\frac{d}{dt}|_{t=0}\Psi_t$
is a an $O(n)$-invariant vector field on $\Cal OM$. Thus $\Ga(\Cal AM)$ can be
interpreted as the space of infinitesimal principal bundle automorphisms of $\Cal
OM$.

Next, by definition of a Cartan connection, the difference $\widehat{\om}-\om$ of two
Cartan connections $\om,\widehat{\om}\in\Om^1(\Cal OM,\frak g)$ is equivariant and
vanishes on fundamental vector fields, so
$\widehat{\om}-\om\in\Om^1_{\hor}(\Cal OM,\frak g)^{O(n)}$. Conversely, for a Cartan
connection $\om$ and $\tau\in\Om^1_{\hor}(\Cal OM,\frak g)^{O(n)}$, also $\om+\tau$
is a Cartan connection, provided that it restricts to a linear isomorphism on each
tangent space.  Thus, Cartan connections form an open subset in an affine space
modeled on $\Om^1_{\hor}(\Cal OM,\frak g)^{O(n)}$. In particular, for an open interval
$I\subset\Bbb R$ with $0\in I$ and a smooth family $\{\om_t:t\in I\}$ of Cartan
connections, we can (point-wise) form a derivative
$\frac{d}{dt}|_{t=0}\om_t\in\Om^1_{\hor}(\Cal OM,\frak g)^{O(n)}$. Hence we can view
$\Om^1_{\hor}(\Cal OM,\frak g)^{O(n)}$ as the space of infinitesimal deformations of
a Cartan connection $\om$.

In degree two, the interpretation is even easier, since the curvature of a Cartan
connection by definition lies in $\Om^2_{\hor}(\Cal OM,\frak g)^{O(n)}$, so for a
smooth family $\om_t$ with curvatures $K_t$ we can form $\frac{d}{dt}|_{t=0}K_t$ to
obtain an element in this space.

The interpretation of the first operator in the sequence follows readily from part
(2) of Theorem \ref{thm3.3}:  For a vector field $\xi\in\frak X(\Cal OM)$ with local
flows $\Fl^{\xi}_t$, it is well known that $\frac{d}{dt}|_{t=0}(\Fl^\xi_t)^*\om=\Cal
L_{\xi}\om$. Hence the first operator $\widetilde{d}^{\om}$ in the sequence computes
the infinitesimal change of Cartan connection $\om$ caused by an infinitesimal
automorphism of the principal bundle $\Cal OM$. In particular, its kernel consists
exactly of the infinitesimal automorphisms of the Cartan geometry $(\Cal OM,\om)$.

The interpretation of the second operator involves a subtlety, however, which was
already observed in \cite{deformations}: For a smooth family $\om_t$ of Cartan
connections with curvatures $K_t$, we want to compute $\frac{d}{dt}|_{t=0}K_t$ from
$\frac{d}{dt}|_{t=0}\om_t$. We have to take into account, however, that the
interpretation of horizontal equivariant forms on $\Cal OM$ as geometric objects on
$M$ involves a Cartan connection. So the right move is to convert $K_t$ to an object
on $M$ using $\om_t$, then apply $\frac{d}{dt}|_{t=0}$ and then convert the result
back using $\om_0$. Let us phrase things in the picture of geometric objects on $M$.

\begin{prop}\label{prop3.4}
  Let $(M,g)$ be a Riemannian manifold with orthonormal frame bundle $\Cal OM$ an let
  $\om$ be the Cartan connection on $\Cal OM$ defined by the Levi-Civita connection
  as in Theorem \ref{thm3.2}.

  (1) Viewing a section $s\in\Ga(\Cal AM)$ as an infinitesimal principal bundle
  automorphism of $\Cal OM$, the induced infinitesimal change of $\om$ corresponds to
  $\tn^{\Cal A}s\in\Om^1(M,\Cal AM)$.

  (2) Viewing $\ph\in\Om^1(M,\Cal AM)$ as an infinitesimal deformation of the Cartan
  connection $\om$, the induced infinitesimal change of curvature corresponds to
  $d^{\tn^{\Cal A}}\ph\in\Om^2(M,\Cal AM)$.
\end{prop}
\begin{proof}
(1) is just a reformulation of the observations on $\widetilde{d}^\om$ in degree zero
  made above.

(2) The definition of the curvature in \eqref{K-def} reads as $K=d^\om \om$, so given
  a family $\om_t$ with $\om_0=\om$, and $\xi,\eta\in\frak X(\Cal OM)$, we have to
  consider
  \begin{equation}\label{Kt}
    K_t:=d\om_t(\xi,\eta)+[\om_t(\xi),\om_t(\eta)].
  \end{equation}
  Defining $\widetilde{\ph}:=\frac{d}{dt}|_{t=0}\om_t\in\Om^1_{\hor}(\Cal OM,\frak
  g)^{O(n)}$, the derivative of \eqref{Kt} with respect to $t$ at $t=0$, clearly is
  $$
  d\widetilde{\ph}(\xi,\eta)+[\widetilde{\ph}(\xi),\om(\eta)]+[\om(\xi),\widetilde{\ph}(\eta)]=
  d^\om\widetilde{\ph}(\xi,\eta).
  $$
  As discussed above, there is the issue of which Cartan connection is used for
  conversion to objects on $M$. Hence rather than the family $K_t$ of two forms, we
  should consider the family of equivariant functions $\Cal OM\to L(\La^2\Bbb
  R^n,\frak g)$ given by $(X,Y)\mapsto K_t((\om_t)^{-1}(X),(\om_t)^{-1}(Y))$, which
  equivalently describes the family of forms in $\Om^2(M,\Cal AM)$. Applying
  $\frac{d}{dt}|_{t=0}$ to this instead of \eqref{Kt}, we get two additional
  summands, given by
  $$
  K_0((\tfrac{d}{dt}|_{t=0}(\om_t)^{-1})(X),(\om_0)^{-1}(Y))+K_0((\om_0)^{-1}(X),
  (\tfrac{d}{dt}|_{t=0}(\om_t)^{-1})(Y)).
  $$ Differentiating the equation $\xi=(\om_t)^{-1}(\om_t(\xi))$ with respect to $t$
  at $t=0$, we get
  $0=(\frac{d}{dt}|_{t=0}(\om_t)^{-1})(\om(\xi))+(\om_0)^{-1}(\widetilde{\ph}(\xi))$. This
  shows that the necessary correction to be added to $d^\om\widetilde{\ph}$ is
  exactly $-i_{\widetilde{\ph}}K$. Thus the claim follows from the relation between
  $\widetilde{d}^\om$ and $d^{\tn^{\Cal A}}$.
\end{proof}

\subsection{Harmonic subbundles and Lie algebra cohomology}\label{3.5}
We next discuss how to conceptually obtain the harmonic subbundles $\Cal H^k$
introduced in Section \ref{2.4}, which brings algebra into the game. Viewing $\Bbb
R^n$ as a linear subspace of $\euc(n)$, we can restrict the Lie bracket of $\euc(n)$
to a bilinear map $\Bbb R^n\x\euc(n)\to\euc(n)$, which is explicitly given by
$(v,(Y,w))\mapsto (0,-Yv)$, so it just encodes the standard action of $\frak o(n)$ on
$\Bbb R^n$. If follows from general principles that this defines a representation of
the Abelian Lie algebra $\Bbb R^n$ on the vector space $\euc(n)$ and that it is
$O(n)$-equivariant. Still by general principles, this gives rise to a sequence of Lie
algebra cohomology differentials $\partial:\La^k(\Bbb
R^n)^*\otimes\euc(n)\to\La^{k+1}(\Bbb R^n)^*\otimes\euc(n)$ defined by
\begin{equation}\label{part-def}
\partial\al(v_0,\dots,v_k):=\textstyle\sum_{i=0}^k(-1)^i[v_i,\ph(v_0,\dots,\widehat{v_i},\dots,v_k)],
\end{equation}
with the hat denoting omission. In particular, these maps also are $O(n)$-equivariant
and hence induce bundle maps $\partial:\La^kT^*M\otimes\Cal
AM\to\La^{k+1}T^*M\otimes\Cal AM$ for $k=0,\dots n-1$, which also satisfy
$\partial\o\partial=0$. Definition \ref{def2.2}, Proposition \ref{prop2.3} and
formula \eqref{Alt-def} then readily imply that $d^{\nabla^{\Cal
    A}}=d^\nabla+\partial$, where we also denote by $\partial$ the tensorial
operation on sections induced by the bundle map. On the other hand, the discussion in
Section \ref{2.4} can be interpreted as saying that $\Cal
H^i=\ker(\partial)/\im(\partial)$ for each $i=0,\dots,n$. So the bundle $\Cal H^i$ is
the associated vector bundle induced by the Lie algebra cohomology space $H^i(\Bbb
R^n,\euc(n))$ for each $i$.

\begin{remark}\label{rem3.5}
(1) Observe that the degree-one component of $\partial$ (on the Lie algebra level) is
  a map $(\Bbb R^n)^*\otimes\frak o(n)\to\La^2(\Bbb R^n)^*\otimes \Bbb R^n$. By
  definition, this is exactly the \textit{Spencer differential} that arises in the
  interpretation of Riemannian metrics as a G-structure with structure group
  $O(n)$. The fact that this is a linear isomorphism is the only specific ingredient
  needed to show that Riemannian metrics admit an equivalent encoding as a Cartan
  geometry. Also many other proofs of existence and uniqueness of the Levi-Civita
  connection can be traced back to this fact. This can be viewed as one step in a
  non-linear version of the sequence we construct here.

  (2) For an arbitrary representation $V$ of $\Euc(n)$, we can take the infinitesimal
  representation of $\euc(n)$ and restrict it to the Abelian subalgebra $\Bbb
  R^n$. This makes $V$ into a representation of $\Bbb R^n$ and by construction, the
  action $\Bbb R^n\x V\to V$ is $O(n)$-equivariant. This leads to a sequence of
  $O(n)$-equivariant Lie algebra cohomology differentials as above and Lie algebra
  cohomology spaces $H^*(\Bbb R^n,V)$. It is well known that $\Euc(n)$ can be
  realized as a matrix group by mapping $(A,v)$ to matrix $\begin{pmatrix} A & v\\ 0
    & 1 \end{pmatrix}$ which defines a representation on $\Bbb R^{n+1}$. Restricted
  to the subgroup $O(n)$, this is isomorphic to $\Bbb R^n\oplus\Bbb R$ with the
  standard representation on the first summand and a trivial representation on the
  second summand. This gives rise to a representation on $\La^2(\Bbb R^{n+1})^*$,
  whose restriction to $O(n)$ decomposes as $\La^2(\Bbb R^n)^*\oplus(\Bbb R^n)^*$
  with the natural representations on both factors. Interpreting the components as
  (multi-)linear maps on $\Bbb R^n$, the action of $\Bbb R^n$ is given by $v\cdot
  (\al,\la)=(0,\al(v,\_))$. This implies that there is an $O(n)$-equivariant linear
  isomorphisms $\euc(n)\to \La^2(\Bbb R^{n+1})^*$, which is compatible with the
  actions of $\Bbb R^n$ on the two spaces. Hence we conclude that the cohomology
  spaces associated to the two representations are isomorphic as representations of
  $O(n)$.

  But the action of $\Bbb R^n$ on $\La^2(\Bbb R^{n+1})^*$ evidently is equivariant
  for the bigger group $SL(n,\Bbb R)$, so the same is true for the induced Lie
  algebra cohomology differentials and hence the cohomology spaces carry natural
  representations of $SL(n,\Bbb R)$. This explains why the cohomology spaces are
  isomorphic to $SL(n,\Bbb R)$ invariant subspaces and hence often not irreducible as
  representations of $O(n)$. This observation is also the basis for the BGG
  construction of projectively invariant versions of the deformation sequence
  mentioned in the Introduction. However, these constructions miss the relation to
  Cartan geometries and hence the natural interpretation in terms of deformations.
\end{remark}

\subsection{Passing to the underlying geometry}\label{3.5a}
The next step is to relate the deformation picture on the Cartan level to the
harmonic subbundles. This is easy in degree zero, where $\Cal H^0=TM$, so sections
are vector fields, which can be viewed as infinitesimal diffeomorphisms of $M$. Thus
the tensorial projection $\Ga(\Cal AM)\to\frak X(M)$ is just the infinitesimal analog
of mapping an automorphism of $\Cal OM$ to its base map, which is a diffeomorphism of
$M$.

To discuss degree one, we know from Section \ref{3.1} that a Cartan connection
$\om_t$ on the principal $O(n)$-bundle $\Cal OM$ equivalently encodes a Riemannian
metric $g_t$ on $M$ and a linear connection $\nabla^t$ on $TM$, which is metric for
$g_t$. Now Riemannian metrics are sections of an open subbundle of $S^2T^*M$ it is no
problem to form the corresponding infinitesimal deformation of $g$ as
$\frac{d}{dt}|_{t=0}g_t$, which is a section of $S^2T^*M=\Cal H^1$. One could also
look at the infinitesimal change of the connection, but we don't have to do this.

We just observe that the connection $\nabla^t$ has a torsion and a curvature, which
clearly both depend smoothly on $t$ and we want to consider their infinitesimal
change. This is easy for the torsion, which is an element of $\Om^2(M,TM)$, so the
infinitesimal change also sits in this space, For the curvature there is an
additional issue, which connects to the discussion on the principal bundle level
before Proposition \ref{prop3.4}. The most natural interpretation of $R$ is as an
element of $\Om^2(M,\frak o(TM))$, which corresponds to the abstract index expression
$R_{ij}{}^k{}_\ell$. But here the subbundle $\frak o(TM)\subset T^*M\otimes TM$
depends on the metric $g$, so since the connections $\nabla^t$ are not metric with
respect to $g=g_0$, their curvature will not have values in $\frak o(TM)$. Hence one
could only hope to get an infinitesimal deformation as an element of
$\Om^2(M,T^*M\otimes TM)$, which is not what one wants in this situation. To deal
with this problem, a similar move as in Section \ref{3.4} is needed. One first lowers
the upper index using the metric $g_t$, i.e.~passes to $R^t_{ijk\ell}$ using this
metric. This then is a section of $\Om^2(M,\La^2T^*M)$ for all $t$, so one can apply
$\tfrac{d}{dt}|_{t=0}$ to get an element of that space. This then has to be converted
back to an element of $\Om^2(M,\frak o(TM))$ using $g=g_0$.
\begin{prop}\label{prop3.5}
Let $(M,g)$ be a Riemannian manifold with equivalent Cartan geometry $(p:\Cal OM\to
M,\om)$ and let us view $\binom{\psi}{\Psi}\in\Om^1(M,\Cal AM)$ as an infinitesimal
deformation of the Cartan connection $\om$.

(1) The infinitesimal deformation $h\in\Ga(\Cal H^1)$ of $g$ corresponding to
$\binom{\psi}{\Psi}$ is given by
$h(\xi,\eta)=g(\psi(\xi),\eta)+g(\xi,\psi(\eta))$. As an endomorphism of $TM$, this
equals $\psi+\psi^t$, where the transpose is taken with respect to $g$.

(2) Writing $d^{\tn^{\Cal A}}\binom{\psi}{\Psi}=\binom{\tau}{\rho}$, the two
component-forms $\tau\in\Om^2(M,TM)$ and $\rho\in\Om^2(M,\frak o(TM))$ describe the
infinitesimal changes of the torsion and of the curvature of the connection part of
the infinitesimal deformation of $\om$, respectively.
\end{prop}
\begin{proof}
(1) Consider a family $\om_t$ of Cartan connections with $\om_0=\om$ and decompose
  $\om_t=\th_t\oplus\ga_t$ into components in $\Bbb R^n$ and $\frak o(n)$. As we have
  noted already, $\om_t-\om_0\in\Om^1_{\hor}(\Cal OM,\euc(n))^{O(n)}\cong\Om^1(M,\Cal
  AM)$. The component $\mu_t$ of this in $\Om^1(M,TM)$ by construction has the
  property that $\frac{d}{dt}|_{t=0}\mu_t=\psi$ and we note that $\mu_t$ corresponds
  to $\th_t-\th_0$.

Now for a point $x\in M$ choose $u\in \Cal OM$ with $p(u)=x$, so $u$ is a linear
isomorphism $\Bbb R^n\to T_xM$ which is orthogonal for the standard inner product and
$g(x)=g_0(x)$. For tangent vectors $v,w\in T_xM$ and lifts
$\widetilde{v},\widetilde{w}\in T_u\Cal OM$ the metric $g_t$ induced by $\om_t$ by
definition is given by
$g_t(v,w)=\langle\th_t(\widetilde{v}),\th_t(\widetilde{w})\rangle$, and again by
definition $h(v,w)$ is the derivative with respect to $t$ at $t=0$ of this function.

For smooth curves $a_t,b_t$ of vectors in $\Bbb R^n$, writing $a_t=a_0+(a_t-a_0)$ and
likewise for $b$ immediately shows that $\langle a_t,b_t\rangle$ can be written as
$$
\langle a_0,b_0\rangle+\langle a_t-a_0,b_0\rangle+\langle a_0,b_t-b_0\rangle+\langle
a_t-a_0,b_t-b_0\rangle. 
$$
Now the first summand is constant, while in the last summand both entries of the
inner product vanish for $t=0$, so to compute the derivative at $t=0$ of $\langle 
a_t,b_t\rangle$, we may leave out these two summands. Hence we can compute $h(v,w)$
as the derivative at $t=0$ of
\begin{equation}\label{curve}
\langle(\th_t-\th_0)(\widetilde{v}),\th_0(\widetilde{w})\rangle+
\langle\th_0(\widetilde{v}),(\th_t-\th_0)(\widetilde{w})\rangle. 
\end{equation}
The correspondence between $\mu_t$ and $\th_t-\th_0$ is given by
$\mu_t(v)=u^{-1}((\th_t-\th_0)(\widetilde{v}))\in T_xM$. Hence we can rewrite
\eqref{curve} as $g(\mu_t(v),w)+g(v,\mu_t(w))$, which immediately implies the claim.

(2) Taking into the account the subtleties in both pictures as discussed above, this
becomes a re-interpretation of part (2) of Proposition \ref{prop3.4}.
\end{proof}

\subsection{Interpretation of the BGG machinery}\label{3.6}
To interpret the remaining developments in Section \ref{2} in the picture of
deformations, we need two more concepts.
\begin{definition}\label{def3.5}
  Let $(M,g)$ be a Riemannian manifold with equivalent Cartan geometry $(p:\Cal OM\to
  M,\om)$ and let $\binom{\ps}{\Psi}\in\Om^1(M,\Cal AM)$ be an infinitesimal
  deformation of $\om$.

(1) We call the deformation \textit{symmetric} if and only if the component
  $\psi\in\Om^1(M,TM)$ is symmetric with respect to $g$.

(2) We call the deformation \textit{torsion-free} if and only if the induced
  infinitesimal deformation of the torsion vanishes identically.
\end{definition}

Using this, we can rephrase Theorem \ref{thm2.4} and the definition of the BGG
operators in the following way.

\begin{thm}\label{thm3.6}
  Let $(M,g)$ be a Riemannian manifold and let $(p:\Cal OM\to M,\om)$ be the
  equivalent Cartan geometry.
  
(1) For any vector field $\eta\in\frak X(M)$, $L(\eta)\in\Ga(\Cal AM)$ is the unique
  lift of $\eta$ which defines a symmetric infinitesimal deformation of
  $\om$. Moreover, the resulting infinitesimal deformation of the metric $g$ is given
  by $D^0(\eta)\in\Ga(S^2T^*M)$. In particular, $D^0$ is the Killing operator,
  i.e.~its kernel are the infinitesimal isometries of $(M,g)$.

(2) Let $h\in\Ga(S^2T^*M)$ be an infinitesimal deformation of $g$. Viewing $h$ also
  as an element of $\Cal S(TM)$, there is a unique element
  $L(h)=\binom{h}{\Psi}\in\Om^1(M,\Cal AM)$ which defines a torsion-free
  infinitesimal deformation of $\om$. The resulting deformation of the Riemann
  curvature is then given by $\frac12D^1(h)\in\Om^2(M,\frak o(TM))$. Hence $\frac12
  D^1$ is the operator that computes the infinitesimal change of curvature caused by
  an infinitesimal deformation of the metric $g$.
\end{thm}

Combined with Theorem \ref{thm2.5}, this result can be interpreted as expressing the
consequence of the categorical equivalence between Riemannian $n$-manifolds and
torsion-free Cartan geometries of type $(\Euc(n),O(n))$ for linearized deformation
theory in the two pictures. There is one more interesting observation coming from the
proof of Theorem \ref{thm2.5}. Given an infinitesimal deformation $\binom{\ps}{\Psi}$
of $\om$, we have seen in the proof that the alternating part of $\nabla\ps$ is a
section $\Ph$ of $\frak o(TM)$, so $\binom{0}{\Ph}\in\Ga(\Cal AM)$ can be considered
as a vertical infinitesimal automorphism of $\Cal OM$ (in the sense the it projects
to the zero vector field on $M$. So $d^{\tn^{\Cal A}}\binom{0}{\Ph}$ defines a vertical
trivial infinitesimal deformation of $\om$. Thus the proof of Theorem \ref{thm2.5}
shows that any infinitesimal deformation of $\om$ can be made symmetric by adding a
vertical trivial infinitesimal deformation. 

\subsection{Remarks on general metric connections}\label{3.8}
In the setting of Cartan geometries, it is no problem to study linearized deformation
theory starting with any initial geometry $(\Cal OM\to M,\om)$. The operations
$d^\om$ and $\tilde\xi\mapsto\Cal L_{\tilde\xi}\om$ make sense in this setting and as
in the proof of Theorem \ref{thm3.3}, one shows that they define linear connections
on $\Cal AM$ that can be described explicitly. For the operation induced by $d^\om$
the formula is as in Definition \ref{def2.2} but with the Levi-Civita connection
replaced by the metric linear connection defined by $\om$. For the other operation,
one not only has to move to the that connection and use its curvature, but in addition
one has to add its torsion applied to $\xi$ and $\eta$ in the top row.

Using these formulae, one can compute the curvatures of the connections and the
induced covariant exterior derivatives explicitly. The interpretations of these
operators in terms of the deformation theory of the Cartan connection remain valid in
the case of non-zero torsion. There is no hope, however, to obtain an analog of the
BGG reduction in this more general setting, since neither the original geometry nor
its deformations are determined by some simpler structure. The only potential
simplification could arise from the bijective correspondence between metric
connections and their torsion, but it does not seem obvious how to exploit this.

\end{document}